\documentclass{article}
\usepackage{amssymb,amsmath,amsthm}
\usepackage{fancyhdr}
\usepackage[british]{babel}
\usepackage{enumitem}
\usepackage{lineno}
\usepackage{tikz}
\usepackage{hyperref}
\usepackage[margin=2.5cm]{geometry}

\newtheorem{theorem}{Theorem}[section]
\newtheorem{prop}[theorem]{Proposition}
\newtheorem{lemma}[theorem]{Lemma}

\theoremstyle{definition}

\newtheorem{question}{Question}

\newtheorem{defn}[theorem]{Definition}
\newtheorem{fac}{Fact}

\newtheorem{rmk}{Remark}

\newtheorem{claim}[theorem]{Claim}

\newcommand{\N}{\mathbb{N}}
\newcommand{\eps}{\varepsilon}

\newcommand{\de}{\delta}

\title{Clique-factors in graphs with low $K_{\ell}$-independence number}

\author{
	Ming Chen\thanks{School of Mathematics and Statistics, Jiangsu Normal University, Xuzhou, China. Email: {\tt chenming314@jsnu.edu.cn}. Supported by National Key Research and Development Program of China (2024YFA1013900).}
	\and
	Jie Han\thanks{School of Mathematics and Statistics and Center for Applied Mathematics, Beijing Institute of Technology, Beijing, China. Email: {\tt han.jie@bit.edu.cn}. Supported by National Natural Science Foundation of China (12371341).}
	\and
	Donglei Yang\thanks{School of Mathematics, Shandong University, Jinan, China, Email: {\tt dlyang@sdu.edu.cn}. Supported by Natural Science Foundation of China (12101365) and Natural Science Foundation of Shandong Province (ZR2021QA029).}
}

\begin{document}
\maketitle
%\linenumbers

\begin{abstract}
Given $r\in \mathbb{N}$ with $r\geq 4$, we show that there exists $n_0\in \mathbb{N}$ such that for every $n\geq n_0$, every $n$-vertex graph $G$ with $\delta(G)\geq (\frac{1}{2}+o(1))n$ and $\alpha_{r-2}(G)=o(n)$ contains a $K_{r}$-factor.
This resolves the first open case of a question proposed by Nenadov and Pehova, and reiterated by Knierm and Su.
We further introduce two lower bound constructions that, along with some known results, fully resolve a question presented by Balogh, Molla, and Sharifzadeh.
\end{abstract}
\section{Introduction}
Let $H$ be an $h$-vertex graph and $G$ be an $n$-vertex graph. An $H$-\emph{tiling} in $G$ is a collection of vertex-disjoint subgraphs of $G$ isomorphic to $H$.
An $H$-\emph{factor} is an $H$-tiling which covers all vertices of $G$.
Determining sufficient conditions for the existence of an $H$-factor is one of the fundamental lines of research in extremal graph theory.
One important reason is due to a result of Hell and Kirkpatrick \cite{1983Hell} which shows that the decision problem for $H$-factors is \emph{NP}-complete, given that $H$ has a connected component of size at least 3.

The first result in this line of attack is by Dirac \cite{Dirac1952}, who proved that every $n$-vertex graph $G$ with $\delta(G)\geq\frac{n}{2}$ contains a Hamiltonian cycle, in particular if $n$ is even then $G$ has a perfect matching.
Later a  result of Hajnal and Szemer\'{e}di \cite{HajnalSze1970} gives a sufficient condition $\delta(G)\geq \frac{r-1}{r}n$ guaranteeing the existence of a $K_{r}$-factor.
The $K_{3}$-factor case was previously proved by Corr\'{a}di and Hajnal \cite{MR200185}.
The unbalanced complete $r$-partite graph witnesses the tightness of the minimum degree condition.
Determining the best possible bound on $\delta(G)$ for an arbitrary graph $H$ has been an intriguing problem.
There had been many excellent results (see \cite{MR1767021, MR1829855, MR1995690} and survey \cite{MR2588541}) until it was finally settled by K\"{u}hn and Osthus \cite{DKDO2009}.
There have been various generalizations of Hajnal--Szemer\'{e}di theorem in the setting of partite graphs \cite{MR3354296, KKore2008, MR1910115, MR2433861}, directed graphs \cite{MR3406450} and hypergraphs \cite{MR2500161}.

\subsection{Ramsey--Tur\'{a}n Tiling problems}

Note that the extremal example that achieves the optimality of the bound on $\delta(G)$ in Hajnal--Szemer\'{e}di Theorem contains a large independent set, which is ``regular'' and rather rare among all graphs.
Following the spirit of the well-known Ramsey--Tur\'{a}n theory (see \cite{MR1476449, MR716422, MR0299512, MR2500161}), a natural question on the Hajnal--Szemer\'{e}di Theorem is to determine the minimum degree condition forcing a clique factor when the host graph has sublinear independence number.
The following Ramsey--Tur\'{a}n type question was proposed by Balogh, Molla, and Sharifzadeh \cite{MR3570984}.

\begin{question}[\cite{MR3570984}, Question 1.1]\label{pro1.1}
Let $r\geq 3$ be an integer and $G$ be an $n$-vertex graph with $\alpha(G)=o(n)$. What is the minimum degree condition on $G$ that guarantees a $K_{r}$-factor?
\end{question}

Balogh, Molla, and Sharifzadeh \cite{MR3570984} studied the $K_{3}$-factor, and proved the following result.

\begin{theorem}[\cite{MR3570984}]\label{bms1}
Given constant $\mu >0$, there exists $\alpha >0$ such that the following holds for sufficiently large $n\in 3\mathbb{N}$. Let $G$ be an $n$-vertex graph with $\delta(G)\geq (\frac{1}{2}+\mu)n$ and $\alpha(G)\leq \alpha n$. Then $G$ contains a $K_{3}$-factor.
\end{theorem}

Recently, Knierim and Su \cite{MR4193066} proved the following result, which asymptotically solves the case of larger cliques in Question \ref{pro1.1}.
\begin{theorem}[\cite{MR4193066}]\label{thm1.31}
Given $\mu >0$ and $r\in \mathbb{N}$ with $r\geq 4$, there exists $\alpha >0$ such that the following holds for sufficiently large $n\in r\mathbb{N}$. Let $G$ be an $n$-vertex graph with $r$ dividing $n$, $\delta(G)\geq (1-\frac{2}{r}+\mu)n$ and $\alpha(G)\leq \alpha n$. Then $G$ contains a $K_{r}$-factor.
\end{theorem}

Nenadov and Pehova \cite{MR4080942} proposed an extension of Question \ref{pro1.1} in the following, which was reiterated by Knierm and Su \cite{MR4193066}.
Given a graph $G$ and an integer $\ell\geq 2$, we denote by $\alpha_{\ell}(G)$ the \emph{$K_{\ell}$-independence number} of $G$, which is the maximum size of a $K_{\ell}$-free
subset of vertices in $V(G)$.

\begin{question}[\cite{MR4193066, MR4080942}]\label{pro1.2}
Is it true that for every $r, \ell\in \mathbb{N}$ with $r\geq \ell$ and $\mu>0$ there is a constant $\alpha$ and $n_{0}\in \mathbb{N}$ such that every graph $G$ on $n\geq n_{0}$ vertices where $r$ divides $n$ with $\delta(G)\geq \max\left\{\left(\frac{1}{2}+\mu\right)n, \left(\frac{r-\ell}{r}+\mu\right)n\right\}$ and $\alpha_{\ell}(G)\leq \alpha n$ has a $K_{r}$-factor?
\end{question}

Nenadov and Pehova \cite{MR4080942} also proved a general upper bound for the minimum degree conditions, which is asymptotically optimal when $\ell=r-1$.
In the following, we give an affirmative answer to Question \ref{pro1.2} for the case $\ell=r-2$.
\begin{theorem}[]\label{thm7.2}
Given $r\in \mathbb{N}$ with $r\ge 4$ and $\mu>0$, there exists $\alpha>0$ such that the following holds for sufficiently large $n\in r\mathbb{N}$.
Let $G$ be an $n$-vertex graph such that $\delta(G)\geq \left(\frac{1}{2}+\mu\right)n$ and $\alpha_{r-2}(G)\leq \alpha n$. Then $G$ contains a $K_{r}$-factor.
\end{theorem}
The minimum degree in Theorem \ref{thm7.2} is asymptotically tight: consider an $n$-vertex graph $G$ consisting of two disjoint complete graphs of order $\frac{n}{2}+1$ and $\frac{n}{2}-1$.

There are also other (indeed, negative) results on Question \ref{pro1.2}.
Given integers $n, r$ and a function $f(n)$, we use $RT_{\ell}(n, K_{r}, f(n))$ to denote the maximum number of edges of an $n$-vertex $K_r$-free graph $G$ with $\alpha_{\ell}(G)\leq f(n)$.
In particular, the \emph{Ramsey-Tur\'{a}n density} of $K_r$ is defined as
\begin{center}
$\rho_{\ell}(r):=\lim\limits_{\alpha\rightarrow 0}\lim\limits_{n\rightarrow \infty} \frac{RT_{\ell}(n, K_{r}, \alpha n)}{\tbinom{n}{2}}$.
\end{center}
Chang, Han, Kim, Wang, and Yang \cite{CHKWY} determined that $\delta(G) \ge \left( \frac{1}{2-\varrho_{\ell}(r-1)} + o(1)\right)n $ is asymptotically tight for the case $r > \ell \ge \frac{3}{4}r$.
Comparing with the results in \cite{CHKWY}, Theorem \ref{thm7.2} gives a new result for $r\in \{5,6,7\}$.
Recently, Han, Hu, Wang, and Yang \cite{HHYW} extended the investigation of Question \ref{pro1.2} under the condition $\alpha_{\ell}(G)= n^{1-o(1)}$, and proved that every $n$-vertex graph $G$ with $\delta(G)\geq \max\big\{\frac{r-\ell}{r}+o(1), \frac{1}{2-\rho_{\ell}(r-1)}+o(1)\big\}n$ and $\alpha_{\ell}(G)\leq n^{1-o(1)}$ contains a $K_{r}$-factor.

We remark that it would be interesting to try to solve Question \ref{pro1.2} for the next concrete case $\ell=3$ and $r=6$, which seems to be out of reach via our approaches.

\subsection{$K_r$-factor in $K_s$-free graphs}
Note that the construction of two vertex-disjoint cliques of almost equal size not divisible by $3$ has low
independence number and contains no triangle factors.
In \cite{MR3570984}, Balogh, Molla, and Sharifzadeh suggested the following question which concerns the absence of larger cliques.

\begin{question}[\cite{MR3570984}, Question 6.3]\label{pro1.11}
Let $s, r\in \mathbb{N}$ with $s>r$, and $G$ be an $n$-vertex $K_{s}$-free graph with $\alpha(G)=o(n)$.
What is the minimum degree condition on $G$ that guarantees a $K_{r}$-factor in $G$?
\end{question}

Balogh, Molla, and Sharifzadeh \cite{MR3570984} also established the following two constructions for Question \ref{pro1.11}.
For completeness, we restate the constructions in Section \ref{sec2.1}.
\begin{prop}[\cite{MR3570984}, Example 6.1]\label{prop1.1212}
Given $\alpha, \varepsilon>0$, $r\in \mathbb{N}\backslash 2\mathbb{N}$ (or $r\in 2\mathbb{N}$), and sufficiently large $n$, there exists an $n$-vertex graph with $\delta(G)\geq (1-\frac{2}{r}+\varepsilon)n$ and $\alpha(G)\leq \alpha n$ such that $G$ is $K_{r+2}$-free (resp. $K_{r+1}$) and contains no $K_{r}$-factor.
\end{prop}

%Note that this construction is \red{$K_{r+2}$-free ?} and $K_{r+1}$-free with $r\in 2\mathbb{N}$.
Note that Theorem \ref{thm1.31} together with Proposition \ref{prop1.1212} answered Question \ref{pro1.11} for the cases $s=r+1$ with $r\in 2\mathbb{N}$, and $s\geq r+2$.
The only left case in Question \ref{pro1.11} is $s=r+1$ with $r\in \mathbb{N}\backslash 2\mathbb{N}$.
Here, we resolve this question.
Given $s\geq 3$, let
\[
f_{RT}(s) := \begin{cases} \frac{s-3}{s-1} \quad~~\text{ if s is odd}; \\
\frac{3s-10}{3s-4} \quad \text{ if s is even}.
%2ar(H) \hfill \text{ otherwise}.
\end{cases}
\]
First, for $r=3$ and $s\geq 5$, it is shown by Balogh, McDowell, Molla, and Mycroft \cite{MR3816057} that the minimum degree condition is asymptotically $\frac{n}{3}$ when $s\in \{5, 6\}$, and $\frac{f_{RT}(s)}{2}n$ for $s\geq 7$.
Second, the $r=3$ and $s=4$ case was resolved by Han and Thoma \cite{tohan} (unpublished) that the minimum degree condition is asymptotically $\frac{n}{6}$.

%The proof of Proposition \ref{prop1.1121} will be presented in Section \ref{sec2.1}.
Here, we resolve the remaining case, that is, when $r\geq 5$ is odd and $s=r+1$.
More precisely, we show that the answer is asymptotically $(1-\frac{6}{3r-1})n$ by
%Furthermore, we complete the last piece of Problem \ref{pro1.11},
combining a lower bound construction (see Proposition \ref{prop1.112}), and an upper bound on the Ramsey--Tur\'{a}n density (degree version) for $K_{r+1}$ (see Proposition \ref{cl1.5}).

\begin{prop}\label{prop1.112}
For all $\eps,\alpha>0, r\in \mathbb{N}\backslash 2\N$ with $r\geq 5$, and sufficiently large $n$, there exists an $n$-vertex graph $G$ with $\delta(G)\geq \left(1-\tfrac{6}{3r-1}-\eps\right)n$ and $\alpha(G)\leq \alpha n$ such that $G$ is $K_{r+1}$-free, and without $K_{r}$-factor.
\end{prop}
%The proofs of Proposition \ref{prop1.1121} and Proposition \ref{prop1.112} are similar.
We put the proof of Proposition \ref{prop1.112} in Section \ref{sec2.1}.

The following proposition is an upper bound of Question \ref{pro1.11} for the case $s=r+1$ and $r\in \mathbb{N}\backslash 2\mathbb{N}$.

We include its proof in Section \ref{sec2.1}.

\begin{prop}\label{cl1.5}
Given $r\in \mathbb{N}\backslash 2\N$ and $\eps >0$, there exists $\alpha>0$ such that the following holds for sufficiently large $n$.
Let $G$ be an $n$-vertex $K_{r+1}$-free graph, and $\alpha(G)\le \alpha n$.
Then $\delta(G)<\left(1-\frac{6}{3r-1}+\eps\right)n$.
\end{prop}

It is interesting to look into the critical window and understand the boundary case, e.g., when $\delta(G) \ge (1-\frac{6}{3r-1})n.$

\subsection{Proof strategy}
Our proof makes use of the absorption method and builds on the techniques developed in \cite{MR3529107, MR3632565, MR3290271}.
The absorption method was introduced by R\"{o}dl, Ruci\'{n}ski, and Szemer\'{e}di about a decade ago in \cite{MR2500161}.
Since then, it has turned out to be an important tool for studying the existence of spanning structures in graphs, digraphs and hypergraphs.

Now we sketch the proof idea. The main tasks are to (i) build an absorbing set and (ii) cover almost all of the remaining vertices with a $K_{r}$-tiling.
For (i), we first need the following notion of absorbers and absorbing sets in \cite{MR4080942}.
Let $G$ be an $n$-vertex graph. Then

\begin{itemize}
  \item a subset $A\subseteq V(G)$ is a $\xi$-\emph{absorbing set} in $G$ for some constant $\xi$ if for any subset $U\subseteq V(G)\backslash A$ of size at most $\xi n$ and $|A\cup U|\in r\mathbb{N}$, $G[A\cup U]$ contains a $K_{r}$-factor.
  \item for any $S\subseteq V(G)$ of size $r$ and an integer $t$, we call a subset $A_{S}\subseteq V(G)\backslash S$ a $(K_{r}, t)$-\emph{absorber} for $S$ if $|A_{S}|=rt$ and both $G[A_{S}]$ and $G[A_{S}\cup S]$ contain a $K_{r}$-factor.
\end{itemize}

%However, as pointed out in \cite{MR3570984}, in our setting this is usually impossible because when we construct the absorbers using the independence number condition, it does not give such a strong counting. Instead, a
A new approach due to Nenadov and Pehova \cite{MR4080942} guarantees an absorbing set provided that
\begin{quote}
\emph{every $r$-set $S$ has $\Omega(n)$ vertex-disjoint $(K_{r}, t)$-absorbers for $S$.}
\end{quote}
%However, it is still unclear how to verify this for our problem and the bulk of the work on building absorbing sets is to handle this.
Here we instead build a partition $V(G)=B\cup U$ satisfying that
\begin{quote}
\emph{$|B|=o(n)$ and every $r$-set in $U$ has $\Omega(n)$ vertex-disjoint absorbers.}
\end{quote}
Similar ideas already appeared in our recent works (e.g.~\cite{CHKWY, CHWY202207}).
Here the arguments reduce to finding an absorbing set $A$ in $G[U]$ and then covering vertices of $B$ by a small $K_{r}$-tiling vertex disjoint from $A$ so as to yield a desired absorbing set.
To achieve the desired partition $B\cup U$, our proof adopts the lattice-based absorption method developed by the second author \cite{MR3632565}.\medskip

Our second task is to establish an almost perfect $K_r$-tiling.
Following a standard application of the regularity lemma, we obtain an $\varepsilon$-regular partition and then build a reduced multigraph where two vertices are connected by a double-edge if the corresponding clusters form a regular pair of density larger than $\frac{1}{2}$.
We find that in the reduced multigraph, $K^{=}_2$ and $K_3$ can be used to model how $K_r$ is embedded into a collection of clusters with sublinear $K_{r-2}$-independence number.
To play with this, we build a $\{K^{=}_2, K_3\}$-tiling (see Lemma~\ref{lem4.4}) that covers almost all vertices in the reduced multigraph.
%To apply Lemma~\ref{lem4.4} in our proof, it suffices to show that for every $H\in \{K^{=}_2, K_3\}$ in the reduced multigraph with $V(H)=\{V_1, \ldots, V_i\}$, we can find an almost perfect $K_{r}$-tiling in an arbitrary collection of subclusters $V'_j\subseteq V_j$ (of considerable size) for every $j\in[i]$.
Meanwhile, for every $H\in \{K^{=}_2, K_3\}$, we define a suitably-chosen auxiliary graph $Q_i$ for $i\in \{1, 2\}$.
In particular, $Q_i$ has a $K_{r}$-factor, and thus it suffices to find an almost perfect $\{Q_1, Q_2\}$-tiling in the same context (see Lemma~\ref{lem2.17} and Lemma~\ref{lem2.171}).

Note that Knierm and Su \cite{MR4193066} also used $K^{=}_2$ and $K_3$ to model how $K_{4}$ is embedded into a graph $G$ when $\alpha(G)=o(n)$.
Within this framework, they established a fractional analogue of Lemma \ref{lem4.4} by utilizing the Bollob\'{a}s--Erd\H{o}s graph via an elegant proof.
However, since the Ramsey--Tur\'{a}n problem is still not well understood for $K_{\ell}$-independence number, it is currently unclear to us how to adapt the ideas of Knierm and Su into our proof.
Luckily, in our (multi-graph) setting, we can find an almost perfect tiling with tiles $K^{=}_2$ and $K_3$ directly (instead of their fractional version) under the $\alpha_{r-2}(G)=o(n)$ for $r\geq 4$.

\subsection{Basic notation and organization}
We first introduce some notion throughout the paper.
For a graph $G:= G(V, E)$, we write $v(G)=|V(G)|$ and $e(G)=|E(G)|$. For $U\subseteq V$, let $G[U]$ be the induced subgraph of $G$ on $U$.
Let $G-U:=G[V\backslash U]$.
For two subsets $A, B\subseteq V(G)$, we use $E_G(A, B)$ to denote the set of edges joining $A$ and $B$, and $e_G(A, B):=|E_G(A, B)|$.
We write $N_{G}(v)$ for the set of neighbors of $v$ in $G$.
For convenience, we use $d_{G}(v)$ to denote the number of edges which contain $v$ in $G$. We omit the subscript $G$ if the graph is clear from the context.
For a vertex set $W$ and a positive integer $\ell$, we write $\tbinom{W}{\ell}$ to denote the set of all $\ell$-\emph{subsets} of distinct elements of $W$.
For any integers $a\leq b$, let $[a, b]:=\{i\in \mathbb{Z}: a\leq i\leq b\}$ and $[a]:= [1, a]$.

When we write $\beta\ll \gamma$, we always mean that $\beta, \gamma$ are constants in $(0, 1)$, and $\beta\ll \gamma$ means that there exists $\beta_{0}=\beta_{0}(\gamma)$ such that the subsequent arguments hold for all $0<\beta\leq \beta_{0}$. Hierarchies of other lengths are defined analogously.

The rest of the paper is organized as follows.
In Section \ref{sec2.1}, we introduce three lower bound constructions (Proposition \ref{prop1.1212}, \ref{prop1.112}, and \ref{cl1.5}).
In Section \ref{sec2}, we prove our second result and introduce the regularity lemma.
In Section \ref{sec3}, our main work is to find a tiling with certain structures in the reduced multigraph and establish two embedding lemmas.
%We will finish the proof of Lemma \ref{lem3.1} in Section \ref{sec4}.
In Section \ref{sec5}, we briefly present some necessary results and tools to introduce the latticed-based absorbing method and prove Lemma \ref{lem3.3}.

\section{Proofs of Proposition \ref{prop1.1212}, \ref{prop1.112}, and \ref{cl1.5}}\label{sec2.1}
%We first give some results which are commonly used in the proof of Proposition \ref{prop1.1121} and Proposition \ref{prop1.112}.
Before the proof of Proposition \ref{prop1.1212}, we first introduce a useful construction by Erd\H{o}s \cite{ER1959}.
\begin{lemma}[\cite{ER1959}]\label{lem2.122223}
For any $\alpha, \varepsilon>0$, there exists an $n$-vertex triangle-free graph $G$ for sufficiently large $n$ such that $\alpha(G)<\alpha n$ and $\frac{\varepsilon}{2}n\leq d(v)\leq \frac{3\varepsilon}{2}n$.
\end{lemma}
We denote the construction as above by $\mathbb{ER}(n, \alpha, \varepsilon)$.

%In the following, we call an $n$ vertex triangle-free graph with independence number $o(n)$ and
%minimum degree $o(n)$ an \emph{Erd\H{o}s graph}, which we denote by $\mathbb{ER}(n)$.
\begin{proof}[Proof of Proposition \ref{prop1.1212}]
Given $\alpha, \varepsilon>0$ and $r\in \mathbb{N}$, let $\mu:=2r\varepsilon$.
We shall choose $\frac{1}{n}\ll\alpha, \varepsilon, \mu$.
We divide the proof into the following two cases.
Let $r:=2\ell+1$ for some $\ell\in \mathbb{N}$.
Let $G$ be a complete $(\ell+1)$-partite graph on $n$ vertices, where $V_0$ has size $\frac{n}{r}-1$, $V_1$ has size $\frac{2n}{r}+1$, and the remaining parts $V_2, \dots, V_{\ell}$ each have size $\frac{2n}{r}$.
To complete the construction, we shall put a copy of $\mathbb{ER}(|V_i|, \alpha, \mu)$ on the vertex subset $V_i$ for each $i=0, 1, \dots, \ell$.
We claim that this graph does not have a $K_r$-factor.
Indeed, each $K_r$ has at most two vertices in $V_1$, and a $K_r$-tiling can have at most $\frac{n}{r}$ vertex-disjoint copies of $K_r$.
The minimum degree of this construction is
$n-|V_1|+\frac{\mu}{2}|V_1|\geq (1-\frac{2}{r}+\frac{\mu}{2r})n=(1-\frac{2}{r}+\varepsilon)n$, and $\alpha(G)\leq \alpha n$.
Not that this construction is also $K_{r+2}$-free.
For the $r=2\ell$ case, let $G$ be a complete $\ell$-partite graph on $n$ vertices, where $V_1$ has size $\frac{2n}{r}+1$, $V_2$ has size $\frac{2n}{r}-1$, and the remaining parts $V_3, \dots, V_{\ell}$ each have size $\frac{2n}{r}$.
For each $i\in [\ell]$, we put a copy of $\mathbb{ER}(|V_i|, \alpha, \mu)$ on the vertex subset $V_i$.
Again, this construction has minimum degree at least $(1-\frac{2}{r}+\frac{\mu}{2r})n=(1-\frac{2}{r}+\varepsilon)n$, $\alpha(G)\leq \alpha n$, and no $K_r$-factor.
Not that this construction is also $K_{r+1}$-free.
\end{proof}

Before the proof of Proposition \ref{prop1.112}, we first introduce some results which are useful in the following constructions.
\begin{lemma}[\cite{ER1959}]\label{lem2.1}
For any $\alpha>0$ and $k\in \mathbb{N}$ with $k\geq 3$, there exists an $n$-vertex graph $G$ for sufficiently large $n$ such that $\alpha(G)<\alpha n$ and $g(G)>k$.
\end{lemma}

The following celebrated constructions and claims are proved by Bollob\'{a}s and Erd\H{o}s \cite{boer1976}.

\begin{lemma}[\cite{boer1976}, Bollob\'{a}s--Erd\H{o}s construction]\label{cst1}
For all $\eps,\gamma>0$ there is a $n_0$ such that for $n\ge n_0$ there is a graph $G$ on $2n$ vertices with a balanced partition $V(G)=V_1\cup V_2$ satisfying the following properties:

\item $({\rm 1})$ $G[V_1]$ is isomorphic to $G[V_2]$ and they are triangle-free;\\

\item $({\rm 2})$ $G$ is $K_4$-free;\\

\item $({\rm 3})$ $G[V_1, V_2]$ has density at least $1/2-\eps$;\\
\item $({\rm 4})$ $\alpha(G)\le \gamma n$.

\end{lemma}
More often we denote the construction as above by $\mathbb{BE}(n, \gamma, \varepsilon)$.
We also tacitly assume that $n$ is large, and we omit floors and ceilings, when they do not affect our arguments.

\begin{proof}[Proof of Proposition \ref{prop1.112}]
Given $\eps,\alpha>0$, $r\in \mathbb{N}\backslash2\mathbb{N}$ and $r\geq 5$, we choose $\frac{1}{n}\ll \alpha,\eps$.
Let $G$ be an $n$-vertex graph with a partition $V(G)=V_{1}\cup V_{2}\cup \cdots \cup V_{\frac{r+1}{2}}$, $|V_{1}|=|V_{2}|=\frac{4}{3r-1}n$ and $|V_{i}|=\frac{6}{3r-1}n$ for $i\in [3, \frac{r+1}{2}]$.
Let $G[V_{i}, V_{j}]$ be a complete bipartite graph if either $i\notin [2]$ or $j\notin [2]$.
Let $G[V_{i}]$ be a subgraph with $\alpha(G[V_{i}])\leq \alpha n$ and $g(G[V_{i}])\geq 4$ given by Lemma \ref{lem2.1} for each $i\in [3, \frac{r+1}{2}]$.
Let $G[V_{1}\cup V_{2}]$ be a Bollob\'{a}s--Erd\H{o}s graph  given by Lemma~\ref{cst1}. As $G[V_{1}\cup V_{2}]$ is $K_4$-free and each $G[V_i]$ is triangle-free, one can easily observe that $G$ is $K_{r+1}$-free and $\de(G)\ge n-|V_{1}|-|V_{2}|+(\tfrac{1}{2}-\eps)|V_{1}|\geq(1-\frac{6}{3r-1}-\eps)n$ and $\alpha(G)\le \alpha n$.

Suppose for a contradiction that $G$ contains a $K_{r}$-factor, say $\mathcal{T}$. Then every copy of $K_{r}$ in $\mathcal{T}$, say $H$, satisfies $|V(H)\cap (V_{1}\cup V_{2})|=3$ because each $G[V_i]$ is triangle-free and $G[V_{1}\cup V_{2}]$ is $K_4$-free.
This implies $|V_{1}\cup V_{2}|=3|\mathcal{T}|=\tfrac{3}{r}n>\frac{8}{3r-1}n$, a contradiction.
%In summary, $G$ is an $n$-vertex $K_{r+1}$-free graph without $K_{r}$-factor, $\alpha(G)=o(n)$ and $\delta(G)\geq \left(1-\frac{6}{3r-1}\right)n+o(n)$.
\end{proof}
%In the next section, we prove Theorem \ref{thm7.2}.

The rest of this section was devoted to prove Proposition \ref{cl1.5}.
In the proof, we involve a result of Erd\H{o}s and S\'{o}s \cite{EST1979}.
\begin{theorem}[\cite{EST1979}]\label{thm0.1}
Given $\varepsilon>0$ and $\ell\in 2\mathbb{N}$, there exists an absolute constant $\alpha$ such that, for sufficiently large $n$, $RT(n, K_{\ell}, \alpha n)=\frac{3\ell-10}{6\ell-8}n^{2}+\varepsilon n^{2}$.
\end{theorem}
\begin{proof}[Proof of Proposition \ref{cl1.5}]
Given $r\in \mathbb{N}\backslash 2\N$, we shall choose $\frac{1}{n}\ll \alpha\ll\varepsilon.$
For the sake of contradiction, assume that $\delta(G)\geq \left(\frac{3r-7}{3r-1}+\eps\right)n$.
Then $e(G)\geq \left(\frac{3r-7}{6r-2}+\eps \right)n^{2}>RT(n, K_{r+1}, \alpha n)$.
Applying Theorem \ref{thm0.1} on $G$ with $\ell=r+1\in 2\mathbb{N}$,
there exists a copy of $K_{r+1}$ in $G$, a contradiction.
\end{proof}

\section{Main tools}\label{sec2}

%\subsection{Proof of the main result}\label{sec3}
In this section, we first introduce the central lemmas and explain how they work together to give the proof of Theorem \ref{thm7.2}.
A crucial and necessary step in our proof is to find a $K_{r}$-tiling in the graph $G$ which covers all but a small set of vertices.
The following result guarantees the existence of such a $K_{r}$-tiling, whose proof will be presented in Section \ref{sec3}.

\begin{lemma}[]\label{lem3.1}
Given $r\in \mathbb{N}$ with $r\ge 4$ and $\mu, \delta>0$, there exists $\alpha>0$ such that the following holds for sufficiently large $n$.
Let $G$ be an $n$-vertex graph with $\delta(G)\geq \left(\frac{1}{2}+\mu\right)n$ and $\alpha_{r-2}(G)\leq \alpha n$.
Then $G$ contains a $K_{r}$-tiling that covers all but at most $\delta n$ vertices.
\end{lemma}
%
%Lemma \ref{lem3.3} provides an absorbing set in $G$, whose proof will be given in Section \ref{sec5}.
Lemma \ref{lem3.3} provides an absorbing set in the graph $G$, whose proof can be found in Section \ref{sec5}.

\begin{lemma}[]\label{lem3.3}
Given $r\in \mathbb{N}$ with $r\ge 4$ and $\mu, \gamma$ with $0<\gamma\leq \frac{\mu}{2}$, there exist $\alpha, \xi>0$ such that the following holds for sufficiently large $n$.
Let $G$ be an $n$-vertex graph with $\delta(G)\geq \left(\frac{1}{2}+\mu\right)n$ and $\alpha_{r-2}(G)\leq \alpha n$.
Then $G$ contains a $\xi$-absorbing set $A$ of size at most $\gamma n$.
\end{lemma}

Now we give the proof of Theorem \ref{thm7.2} using Lemma \ref{lem3.1} and Lemma \ref{lem3.3}.

\begin{proof} [Proof of Theorem \ref{thm7.2}] Given $r\in \mathbb{N}$ with $r\ge 4$ and $\mu>0$, we choose
\[\frac{1}{n}\ll \alpha\ll \delta\ll \xi\ll\gamma\ll \mu, \frac{1}{r}.\]
Let $G$ be an $n$-vertex graph with $\delta(G)\geq \left(\frac{1}{2}+\mu\right)n$ and $\alpha_{r-2}(G)\leq \alpha n$. By Lemma \ref{lem3.3} with $\gamma\leq \frac{\mu}{2}$, we find a $\xi$-absorbing set $A\subseteq V(G)$ of size at most $\gamma n$ for some $\xi>0$. Let $G_{1}:= G-A$. Then we have
\begin{center}
$\delta(G_{1})\geq \left(\frac{1}{2}+\mu\right)n-\gamma n\geq \left(\frac{1}{2}+\frac{\mu}{2}\right)n.$
\end{center}
Therefore by applying Lemma \ref{lem3.1} on $G_{1}$ with $\delta>0$, we obtain a $K_{r}$-tiling $\mathcal{T}$ that covers all but a set $L$ of at most $\delta n$ vertices in $G_{1}$. Since $\delta\ll \xi$, the absorbing property of $A$ implies that $G[A\cup L]$ contains a $K_{r}$-factor, which together with $\mathcal{T}$ forms a $K_{r}$-factor in $G$.
\end{proof}

\subsection{Regularity}
To find an almost perfect $K_{r}$-tiling (Lemma \ref{lem3.1}), an important ingredient in our proof is Szemer\'{e}di's Regularity Lemma. In this paper, we make use of a degree form of the regularity lemma \cite{MR1395865}.
Given a graph $G$ and a pair $(V_{1}, V_{2})$ of vertex-disjoint subsets in $V(G)$, the \emph{density} of $(V_{1}, V_{2})$ is defined as

\begin{center}
$d(V_{1}, V_{2})=\frac{e(V_{1}, V_{2})}{|V_{1}||V_{2}|}$.\\
\end{center}

\begin{defn}[]\label{def2.1}
Given $\varepsilon>0$, a graph $G$ and a pair $(V_{1}, V_{2})$ of vertex-disjoint subsets in $V(G)$, we say that the pair $(V_1, V_2)$ is $\varepsilon $-\emph{regular} if for all $X\subseteq V_{1}$ and $Y\subseteq V_{2}$ satisfying\\

\begin{center}
$X \subseteq V_{1}, |X| \ge \varepsilon |V_{1}| $ and $Y \subseteq V_{2}, |Y| \ge \varepsilon |V_{2}|$,
\end{center}
we have
\[
|d(X,Y) - d(V_1,V_2)|  \le  \varepsilon.
\]
\end{defn}

\begin{lemma}[\cite{MR1395865}, Slicing Lemma]\label{lem2.2}
Assume $(V_{1}, V_{2})$ is $\varepsilon$-regular with density $d$. For some $\alpha\geq \varepsilon$, let $V_{1}'\subseteq V_{1}$ with $|V_{1}'|\geq \alpha|V_{1}|$ and $V_{2}'\subseteq V_{2}$ with $|V_{2}'|\geq \alpha|V_{2}|$. Then $(V_{1}', V_{2}')$ is $\varepsilon'$-regular with $\varepsilon':=\max\{2\varepsilon, \varepsilon/\alpha\}$ and for its density $d'$ we have $|d'-d|<\varepsilon$.
\end{lemma}

\begin{lemma}[\cite{MR1395865}, Degree form of the Regularity Lemma]\label{lem2.3}

For every $\varepsilon  > 0$ there is an $N = N(\varepsilon )$ such that the following holds for any real number $\beta\in [0, 1]$ and $n\in \mathbb{N}$. Let $G$ be a graph with $n$ vertices. Then there exists an $(\varepsilon,\beta)$-regular partition $V(G)=V_{0}\cup \cdots \cup V_{k} $ and a spanning subgraph $G' \subseteq G$ with the following properties:\\

         \item $({\rm 1})$ $ \frac{1}{\varepsilon}\leq k \le N $;\\

         \item $({\rm 2})$ $|V_{i}| \le \varepsilon  n$ for $i\in [0, k]$ and $|V_{1}|=|V_{2}|=\cdots=|V_{k}| =m$ for some $m\in \mathbb{N}$;\\

         \item $({\rm 3})$ $d_{G'}(v) > d_{G}(v) - (\beta + \varepsilon )n$ for all $v \in V(G)$;\\

         \item $({\rm 4})$ each $V_{i}$ is an independent set in $G' $ for $i\in [k]$;\\

         \item $({\rm 5})$ all pairs $(V_{i}, V_{j})$ are $\varepsilon $-regular with density 0 or at least $\beta$ for distinct $i, j\neq0$.

\end{lemma}

A widely-used auxiliary graph accompanied with the regular partition is the reduced graph.
To differentiate between dense and very dense pairs of partitions, we employ the following notion of reduced multigraph.

\begin{defn}[Reduced graph]\label{def2.4}
Let $k\in \mathbb{N}$, $\beta, \varepsilon>0$, $G$ be a graph with a vertex partition $V(G)=V_0\cup \ldots \cup V_k$ and $G'\subseteq G$ be a subgraph fulfilling the properties of Lemma \ref{lem2.3}. We denote by $R_{\beta, \varepsilon}$ the \emph{reduced graph} of this partition, which is defined as follows. Let $V(R_{\beta, \varepsilon})=\{V_{1}, \ldots, V_{k}\}$ and for two distinct clusters $V_{i}$ and $V_{j}$ we draw a double-edge between $V_{i}$ and $V_{j}$ if $d_{G'}(V_i, V_j)\geq \frac{1}{2}+\beta$, a single-edge if $\beta\leq d_{G'}(V_i, V_j)<\frac{1}{2}+\beta$ and no edge otherwise.
\end{defn}

The following fact presents a minimum degree of the reduced graph provided the minimum degree of $G$, where a double-edge is counted as two edges.
\begin{fac}[]\label{fact2.5}
Let $n\in \mathbb{N}$, $0<\varepsilon, \beta\leq \frac{\mu}{10}$ with $\beta\in [0, 1]$, and $G$ be an $n$-vertex graph with $\delta(G)\geq(\frac{1}{2}+\mu)n$. Let $V(G)=V_{0}\cup \cdots \cup V_{k}$ be a vertex partition of $V(G)$ satisfying Lemma \ref{lem2.3} (1)-(5). We denote the reduced graph as $R_{\beta, \varepsilon}$. Then for every $V_{i}\in V(R_{\beta, \varepsilon})$ we have

\begin{center}
$d_{R_{\beta, \varepsilon}}(V_{i})\geq(1+\mu)k$.
\end{center}

\end{fac}

\begin{proof}[Proof] Note that $|V_{0}|\leq \varepsilon n$ and $|V_{i}|=m$ for each $i\in [k]$. Every edge in $R_{\beta, \varepsilon}$ represents less than $\left(\frac{1}{2}+\beta\right)m^{2}$ edges in $G'-V_{0}$. Thus we have
\begin{align*}
    d_{R_{\beta, \varepsilon}}(V_{i}) & \geq \frac{|V_{i}|(\delta(G)-(\beta+\varepsilon)n-\varepsilon n)}{\left(\frac{1}{2}+\beta\right)m^{2}} \\
    & \geq \frac{\left(\frac{1}{2}+\mu-2\varepsilon-\beta\right)mn}{\left(\frac{1}{2}+\beta\right)m^{2}} \\
    & \geq 2\left(\frac{1}{2}+\mu-2\varepsilon-\beta\right)(1-2\beta)k\\
    & > (1+\mu)k,
\end{align*}
since $0<\varepsilon, \beta\leq \frac{\mu}{10}$ and $\left(\frac{1}{2}+\beta\right)^{-1}\geq 2(1-2\beta)$.
\end{proof}

\begin{rmk}\label{remk2.6}
Let $R$ be a multigraph with multiplicity 2. We use $K_{2}^{=}$ to denote a copy of double-edge in $R$. Note that $|N(V_i)|\geq \frac{1}{2}d(V_i)$ for each $V_i\in V(R)$. The \emph{double-edge neighborhood} of $V_i\in V(R)$ is a set of vertices in $V(R)$ each of which is connected to $V_i$ through a double-edge. Similarly, we define the \emph{single-edge neighborhood}.
\end{rmk}

\section{Almost perfect tilings}\label{sec3}
To obtain an almost perfect $K_{r}$-tiling in the graph $G$, we first define two suitably-chosen auxiliary graphs $Q_{i}$ (to be defined later) for each $i\in [2]$  such that $Q_{i}$ contains a $K_{r}$-factor.
This roughly reduces the problem to finding in $G$ a collection of vertex-disjoint copies of auxiliary graphs which altogether cover almost all vertices.
Here our proof adopts a standard application the regularity lemma on $G$ to get a reduced graph $R$.
A key step in it is to construct certain structures in $R$ for embedding $Q_{i}$.
In this case, we find two structures, say $K^{=}_2$ and $K_{3}$, and an almost $\{K^{=}_2, K_{3}\}$-tiling in $R$.
Then we develop two tools (see Lemma \ref{lem2.17} and Lemma \ref{lem2.171}) for embedding auxiliary graphs $Q_{1}$ and $Q_{2}$ under certain pseudorandomness conditions.

\subsection{Technical tools}\label{sec2.2}
The main results in this subsection are Lemma \ref{lem4.4} which provides us an almost tiling with two structures in the reduced graph, and two embedding lemmas (Lemma \ref{lem2.17} and Lemma \ref{lem2.171}).
As aforementioned, $K_{2}^{=}$ and $K_{3}$ are two desired structures in the reduced graph.
The following result provides us an almost $\{K^{=}_{2}, K_{3}\}$-tiling in the reduced multigraph.

Here, we want to remark that Knierm and Su \cite{MR4193066} also established a fractional analogue of Lemma \ref{lem4.4}.

\begin{lemma}[]\label{lem4.4}
Given $0<\eta, \mu<1$, the following holds for sufficiently large $k\in \mathbb{N}$. Let $R$ be a $k$-vertex multigraph with multiplicity $2$ and $\delta(R)\geq (1+\mu)k$. Then there exists a $\{K^{=}_{2}, K_{3}\}$-tiling $\mathcal{T}$ in $R$ which covers all but at most $\eta k$ vertices.
\end{lemma}

%By Definition \ref{def2.7}, it holds that $K^{=}_{2}$ and $K_{3}$ are two $K_{4}$-embeddable structures in the reduced graph.
To apply Lemma \ref{lem4.4}, we need to build two intermediate auxiliary graphs $Q_{1}$ and $Q_{2}$ which work for $K^{=}_{2}$ and $K_{3}$ respectively such that each $Q_{i}$ contains a $K_{r}$-factor.

\begin{defn}[]\label{def2.15}
We construct an auxiliary graph $Q_{1}$ with $V(Q_{1})=U_{1}\cup U_{2}$ which satisfies the following conditions:
\begin{itemize}
  \item $Q_{1}$ contains two vertex-disjoint copies of $K_{r}$, say $H_{1}$ and $H_{2}$;
  \item $|V(H_{i})\cap U_{i}|=2$ and $|V(H_{i})\cap U_{i+1}|=r-2$ for $i\in [2]$ (indices $\{1, 2\}$ taken modulo 2).
\end{itemize}
\end{defn}

\begin{defn}[]\label{def2.16}
We construct an auxiliary graph $Q_{2}$ with $V(Q_{2})=U_{1}\cup U_{2}\cup U_{3}$ which satisfies the following conditions:
\begin{itemize}
  \item $Q_{2}$ contains three vertex-disjoint copies of $K_{r}$, say $H_{1}$, $H_{2}$ and $H_{3}$;
  \item $|V(H_{i})\cap U_{i}|=1$, $|V(H_{i})\cap U_{i+1}|=1$ and $|V(H_{i})\cap U_{i+2}|=r-2$ for $i\in [3]$ (indices $\{1, 2\}$ taken modulo 3).
\end{itemize}
\end{defn}

The following two lemmas are two standard gadgets.
%which allow us to embed $Q_{i}$ in $G[V_{1}, \dots, V_{i+1}]$ for every $i\in [2]$.

\begin{lemma}[Embedding $Q_{1}$]\label{lem2.17}
For $r\in \N$ with $r\ge 4$ and $\beta>0$, there exist $\varepsilon, \alpha>0$ such that the following holds for sufficiently large $m$.
Let $G$ be a graph with $V(G)=V_{1}\cup V_{2}$, $\alpha_{r-2}(G)\leq \alpha |V(G)|$ and $|V_{i}|\geq m$ for each $i\in [2]$ such that $(V_{1}, V_{2})$ is $\varepsilon$-regular and $d(V_{1}, V_{2})\geq \frac{1}{2}+\beta$. Then there exists a copy of $Q_{1}$ in $G$ with $V(Q_{1})=U_{1}\cup U_{2}$ and $U_{i}\subseteq V_{i}$ for each $i\in [2]$.
\end{lemma}

\begin{lemma}[Embedding $Q_{2}$]\label{lem2.171}
For $r\in \N$ with $r\ge 4$ and $\beta>0$, there exist $\varepsilon, \alpha>0$ such that the following holds for sufficiently large $m$.
Let $G$ be a graph with $V(G)=V_{1}\cup V_{2}\cup V_{3}$, $\alpha_{r-2}(G)\leq \alpha |V(G)|$ and $|V_{i}|\geq m$ for each $i\in [3]$ such that $(V_{i}, V_{j})$ is $\varepsilon$-regular and $d(V_{i}, V_{j})\geq \beta$ for distinct $i, j\in [3]$.
Then there exists a copy of $Q_{2}$ in $G$ with $V(Q_{2})=U_{1}\cup U_{2}\cup U_{3}$ and $U_{i}\subseteq V_{i}$ for each $i\in [3]$.
\end{lemma}

%In the next subsection, we prove Lemma \ref{lem3.1}.

%\subsection{Proof of Lemma \ref{lem3.1}}\label{sec4}
Equipped with an almost $\{K^{=}_{2}, K_{3}\}$-tiling (Lemma \ref{lem4.4}) in the reduced multigraph and two embedding lemmas (Lemmas \ref{lem2.17} and \ref{lem2.171}), we are able to find an almost perfect $K_{r}$-tiling in the original graph $G$.

\begin{proof} [Proof of Lemma \ref{lem3.1}] Given $r\in \mathbb{N}$ with $r\ge 4$ and positive constants $\delta, \mu$, we shall choose
\begin{center}
$\frac{1}{n}\ll \alpha\ll \frac{1}{k}\ll \varepsilon\ll \beta\ll \eta\ll \delta, \mu, \frac{1}{r}$.
\end{center}
Let $G$ be an $n$-vertex graph with $\delta(G)\geq (\frac{1}{2}+\mu)n$ and $\alpha_{r-2}(G)\leq \alpha n$.
Applying Lemma \ref{lem2.3} with $\varepsilon, \beta>0$, we obtain an $\varepsilon$-regular partition $\mathcal{P}=\{V_{0}, V_{1}, \dots, V_{k}\}$ of $V(G)$.
Let $m:=|V_{i}|$ for each $i\in [k]$, $R:=R_{\beta, \varepsilon}$ be a reduced multigraph of the partition $\mathcal{P}$ with multiplicity $2$ and $V(R)=\{V_{1}, \dots, V_{k}\}$.
By Fact \ref{fact2.5}, we obtain that $\delta(R)\geq (1+\mu)k$.
By applying Lemma \ref{lem4.4} on $R$ with $\eta>0$, we obtain a $\{K^{=}_{2}, K_{3}\}$-tiling $\mathcal{T}$ in $R$, which covers all but at most $\eta k$ vertices.
Let $\mathcal{F}:=\{\mathcal{K}_{1},\dots, \mathcal{K}_{\ell}\}$ be the family of the vertex-disjoint copies of $K^{=}_{2}$ and $K_{3}$ in $\mathcal{T}$.

For every copy of $K^{=}_{2}$ in $\mathcal{T}$, say $\mathcal{K}_{1}$, we assume $V(\mathcal{K}_{1})=\{V_{1}, V_{2}\}$.
Now we construct a $Q_{1}$-tiling $\mathcal{Q}_{\mathcal{K}_{1}}$ by greedily picking vertex-disjoint copies of $Q_{1}$ in $G[V_{1}\cup V_{2}]$ such that $\mathcal{Q}_{\mathcal{K}_{1}}$ is maximal subject to that it contains at most $|V_{i}|$ vertices for each $i\in [2]$.
We claim that $\mathcal{Q}_{\mathcal{K}_{1}}$ covers at least $(1-\varepsilon)m$ vertices from $V_{i}$, $i\in [2]$.
Otherwise, we can apply Lemma \ref{lem2.17} and pick one more copy of $Q_{1}$ in $G[V'_{1}\cup V'_{2}]$ where $V'_{i}$ is an arbitrary vertex subset in $V_{i}$ with $|V'_{i}|\geq \varepsilon m$ for each $i\in [2]$.
This contradicts the maximality of $\mathcal{Q}_{\mathcal{K}_{1}}$.
For every copy of $K_{3}$ in $\mathcal{T}$, say $\mathcal{K}_{2}$, we can apply Lemma \ref{lem2.171}, by the same arguments, and obtain that there exists a $Q_{2}$-tilling $\mathcal{Q}_{\mathcal{K}_{2}}$, which covers all but at most $\varepsilon m$ vertices in each $V_{i}\in V(\mathcal{K}_{2})$.

As the $\{K^{=}_{2}, K_{3}\}$-tiling $\mathcal{T}$ in $R$ covers all but at most $\eta k$ vertices, and $|V_{0}|$ has at most $\varepsilon n$ vertices, we obtain a $(Q_{1}, Q_{2})$-tiling in $G$ which covers at least
\begin{center}
$n-|V_{0}|-\varepsilon mk-\eta km\geq (1-2\varepsilon-\eta)n\geq (1-\delta)n$
\end{center}
vertices, since $\varepsilon\ll \eta\ll \delta$.
%As $\varepsilon\ll \eta\ll \delta$, it holds that $1-2\varepsilon-\eta\geq 1-\delta$.
As each copy of $Q_{i}$ contains a $K_{r}$-factor, the union of these vertex-disjoint $Q_{i}$ provides a $K_{r}$-tiling which covers all but at most $\delta n$ vertices in $G$ and this completes the proof.
\end{proof}
In the next subsection, we prove Lemma \ref{lem4.4}, Lemma \ref{lem2.17} and Lemma \ref{lem2.171}.

\subsection{Proof of related lemmas}\label{sec3.3}
We will end this section with the detailed proofs of Lemmas \ref{lem4.4}, \ref{lem2.17}, \ref{lem2.171}.

\begin{proof}[Proof of Lemma \ref{lem4.4}]
Given $\eta, \mu>0$, we choose $\frac{1}{k}\ll \eta, \mu$.
Let $\mathcal{T}$ be a $\{K^{=}_{2}, K_{3}\}$-tiling in $R$ such that $\mathcal{T}$ is maximal.
%subject to the fact that it contains at most $|V(R)|$ vertices.
%with maximum $|V(\mathcal{T})|$.
Let $A:=V(\mathcal{T})$ and $B:=V(R)\backslash A$.
We may assume $|B|>0$.
Otherwise, we are done.
%Our strategy is to bound $|B|$ from above to obtain an upper bound of $e_R(A, B)$.

Clearly, $R[B]$ is $ K_{3}$-free and $K^{=}_{2}$-free.
Hence, by Tur\'{a}n's theorem, it holds that $e_R(B)\leq \frac{1}{4}|B|^{2}$.
As $\delta(R)\geq (1+\mu)k$, it holds that $e_R(A, B)\geq (1+\mu)k|B|-\frac{1}{2}|B|^{2}$.
%In the following, we obtain an upper bound of $|B|$ by providing an upper bound of $e_R(A, B)$.
Without loss of generality, we assume that $\mathcal{T}$ consists of $p$ vertex-disjoint copies of $K^{=}_{2}$ and $q$ vertex-disjoint copies of $K_{3}$ for some $p, q\in \mathbb{N}$.
Obviously, it holds that $p\leq \frac{|A|}{2}$, $q\leq \frac{|A|}{3}$ and $2p+3q=|A|$.

For any copy of $K^{=}_{2}$ in $\mathcal{T}$, say $uv$, we claim that $N(u)\cap N(v)\cap B=\emptyset$.
%thus $e_R(\{u, v\}, B)\leq 2(|B|-1)+1=2|B|-1$.
Assume $w\in N(u)\cap N(v)\cap B$, we can find a new $\{K^{=}_{2}, K_{3}\}$-tiling $\mathcal{T}'$ by replacing $uv$ with $uvw$.
Clearly, $|V(\mathcal{T}')|=|V(\mathcal{T})|+1$ contradicts the maximality of $|V(\mathcal{T})|$.
Hence, it holds that $e_R(\{u, v\}, B)\leq 2|B|$.

For every copy of $K_{3}$ in $\mathcal{T}$, say $uvw$, we claim that $e_R(\{u, v, w\}, B)\leq 3|B|+3$.
Assume $e_R(\{u, v, w\}, B)\geq 3|B|+4$,
there exists a vertex, say $u$, with $e_R(\{u\}, B)\geq \left\lceil\frac{3|B|+4}{3}\right\rceil=|B|+2$.
This implies that there exist two vertices in $B$, say $u'$ and $u''$, such that $uu'$ and $uu''$ are two double-edges.
We can observe that none element (if exists) in $\{u'v, u'w, u''v, u''w\}$ can be a double-edge, and none element (if exists) in $\{u'vw, u''vw\}$ can be a triangle.
Hence $e_R(\{u, v, w\}, \{u', u''\})\leq 4+2=6$.
For any vertex $z\in B\backslash \{x, y\}$, it holds that $zvw$ can not be a triangle, and none element (if exists) in $\{zv, zw\}$ can be a double-edge.
It holds that $e_R(\{u, v, w\}, \{z\})\leq 3$ and $e_R(\{u, v, w\}, B\backslash\{u', u''\})\leq 3(|B|-2)$.
Now, we have
$e_R(\{u, v, w\}, B)=e_R(\{u, v, w\}, \{u', u''\})+e_R(\{u, v, w\}, B\backslash\{u', u''\})\leq 6+3(|B|-2)=3|B|<3|B|+4$, a contradiction.

Since $2p+3q=|A|$, it holds that
%$(1+\mu)k|B|-\frac{1}{2}|B|^{2} \leq
$e_R(A, B)\leq 2|B|p+(3|B|+3)q\leq 3q+|A||B|$.
Recall that $e_R(A, B)\geq (1+\mu)k|B|-\frac{1}{2}|B|^{2}$.
Let $f:=(1+\mu)k|B|-\frac{1}{2}|B|^{2}-(3q+|A||B|)$, thus $f\leq 0$.
Meanwhile we have
\begin{align*}
       f= & (1+\mu)k|B|-\frac{1}{2}|B|^{2}-(k-|B|)|B|-3q\\
        \geq & \mu k|B|+\frac{1}{2}|B|^{2}-|A|\\
        = & \frac{1}{2}|B|^{2}+k\left(\mu |B|-1\right)+|B|.
\end{align*}
%Since $(1+\mu)k|B|-\frac{1}{2}|B|^{2}\leq e(A, B)\leq 2|B|p+(3|B|+4)q$ and $|B|>0$, it holds that $f<0$.
As $|B|>0$, it holds that $|B|<\frac{1}{\mu}$, and $|B|\leq \eta k$, since $\frac{1}{k}\ll\eta, \mu$.
\end{proof}

\begin{proof}[Proof of Lemma \ref{lem2.17}] Given $ r\in \mathbb{N}$ and $\beta>0$, we choose
\begin{center}
$\frac{1}{m}\ll \alpha\ll \varepsilon\ll \beta$.
\end{center}
Since $(V_{1}, V_{2})$ is $\varepsilon$-regular and $d(V_1, V_2)\geq \frac{1}{2}+\beta$, there exists a subset $V'_{1}\subseteq V_{1}$ with $|V'_{1}|\geq (1-\varepsilon)|V_{1}|$ such that every vertex in $V'_{1}$ has at least $(d(V_{1}, V_{2})-\varepsilon)|V_{2}|\geq (\frac{1}{2}+\frac{\beta}{2})|V_{2}|$ neighbors in $V_{2}$.
Since $\alpha\ll \varepsilon$, it holds that $|V'_{1}|\geq (1-\varepsilon)|V_{1}|\geq \alpha |V(G)|\geq \alpha_{r-2}(G)$.
Hence, there exists a copy of $K_{r-2}$ in $G[V'_{1}]$, say $H_{1}$.
For every edge in $H_1$, say $uv$, $|N_{V_{2}}(u)\cap N_{V_{2}}(v)|\geq \beta|V_{2}|\geq \alpha |V(G)|\geq \alpha_{r-2}(G)$ since $\frac{1}{m}\ll\alpha\ll\varepsilon\ll\beta$.
Then there exists one copy of $K_{r-2}$, say $H_{2}$, in $G[N_{V_{2}}(u)\cap N_{V_{2}}(v)]$.
Till now, we obtain one copy of $K_{r}$, say $H_{3}$, with $|V(H_{3})\cap V_{1}|=2$ and $|V(H_{3})\cap V_{2}|=r-2$.
%Now, we shall put the $H_{3}$ aside.
Let $V''_{1}:= V_{1}\backslash V(H_{3})$ and $V''_{2}:= V_{2}\backslash V(H_{3})$.
By Lemma \ref{lem2.2}, it holds that $(V''_{1}, V''_{2})$ is $\varepsilon'$-regular with $\varepsilon':=\max\left\{2\varepsilon, \frac{\varepsilon|V_{2}|}{|V''_{2}|}\right\}=2\varepsilon$ and $d(V''_{1}, V''_{2})\geq d(V_{1}, V_{2})-\varepsilon\geq \frac{1}{2}+\frac{\beta}{2}$.
In $G[V''_{1}\cup V''_{2}]$, we repeat the upper embedding process and obtain another copy of $K_{r}$, say $H_{4}$, with $|V(H_{4})\cap V''_{1}|=r-2$ and $|V(H_{4})\cap V''_{2}|=2$.
We conclude the proof.
\end{proof}

\begin{proof} [Proof of Lemma \ref{lem2.171}] Given $r\in \mathbb{N}$ and $\beta>0$, we choose
\begin{center}
$\frac{1}{m}\ll \alpha\ll \varepsilon\ll \beta$.
\end{center}
Since $(V_{i}, V_{j})$ is $\varepsilon$-regular for distinct $i, j\in [3]$, there exists a subset $V'_{1}\subseteq V_{1}$ with $|V'_{1}|\geq (1-2\varepsilon)|V_{1}|$ such that every vertex in $V'_{1}$ has at least $(\beta-\varepsilon)|V_{i}|\geq \frac{\beta}{2}|V_{i}|$ neighbors in $V_{i}$ for every $i\in \{2, 3\}$ since $\varepsilon\ll \beta$.
Take any $v\in V'_{1}$, let $V'_{2}:=N(v)\cap V_{2}$ and $V'_{3}:=N(v)\cap V_{3}$. Clearly, $|V'_{i}|\geq \frac{\beta}{2}|V_{i}|$ for $i\in \{2, 3\}$.
By Lemma \ref{lem2.2}, it holds that $(V'_{2}, V'_{3})$ is $\varepsilon'$-regular with $\varepsilon':=\max\left\{2\varepsilon, \frac{\varepsilon|V_{2}|}{|V'_{2}|}\right\}\leq\frac{2}{\beta}\varepsilon$.
Hence, there exists one subset $V''_{2}\subseteq V'_{2}$ such that every vertex in $V''_{2}$ has at least $(d(V_{2}, V_{3})-\varepsilon)|V'_{3}|\geq \frac{\beta}{2}|V'_{3}|$ neighbors in $V'_{3}$.
Take $u\in V''_{2}$, let $V''_{3}:=N(u)\cap V'_{3}$.
Clearly, $|V''_{3}|\geq (\frac{\beta}{2})^{2}|V_{3}|\geq \alpha|V(G)|\geq \alpha_{r-2}(G)$ since $\frac{1}{m}\ll\alpha\ll \varepsilon\ll \beta$.
Then there exists one copy of $K_{r-2}$ in $G[V''_{3}]$.
Till now, we obtain one copy of $K_{r}$, say $H_{1}$ which satisfies $|V(H_{1})\cap V_{1}|=1$, $|V(H_{1})\cap V_{2}|=1$ and $|V(H_{1})\cap V_{3}|=r-2$.
By the similar arguments as above, we can greedily embed the other two desired vertex-disjoint copies of $K_{r}$ in $G$.
We conclude the proof.
\end{proof}

\section{Absorbing}\label{sec5}

In this section, we will prove Lemma \ref{lem3.3}. We use a result of Nenadov and Pehova \cite{MR4080942} which gives a sufficient condition for the existence of an absorbing set.

\begin{lemma}[\cite{MR4080942}, Lemma 2.2]\label{lem5.1}
Let $\gamma>0$ and $r, t\in \mathbb{N}$ be constants.
Then there exists $\xi:=\xi(r, t, \gamma)$ such that the following holds for sufficiently large $n$.
Suppose that $G$ is a graph with $n$ vertices such that every $S\in \binom{V(G)}{r}$ has a family of at least $\gamma n$ vertex-disjoint $(K_r, t)$-absorbers. Then $G$ contains a $\xi$-absorbing set of size at most $\gamma n$.
\end{lemma}

So the key point in the proof of Lemma \ref{lem3.3} is to build linearly many vertex-disjoint absorbers for every $S\in \binom{V(G)}{r}$.
To achieve this, we employ the latticed-based absorbing method \cite{HMWY2021} and we first need the notion of $K_r$-reachability from \cite{HMWY2021} which originates in \cite{MR3338027}.

\begin{defn}
Let $G, K_r$ be given as aforementioned and $m, t\in \mathbb{N}$.
Then we say that two vertices $u, v\in V(G)$ are \emph{$(K_r, m, t)$-reachable} (in $G$) if for any vertex set $W$ of $m$ vertices, there is a set $S\subseteq V(G)\backslash W$ of size at most $rt-1$ such that both $G[S\cup \{u\}]$ and $G[S\cup \{v\}]$ have $K_r$-factors, where we call such $S$ an \emph{$K_r$-connector} for $u, v$.
Moreover, a set $U\subseteq V(G)$ is \emph{$(K_r, m, t)$-closed} if every two vertices $u, v\in U$ are $(K_r, m, t)$-reachable, where the corresponding $K_r$-connector for $u, v$ may not be contained in $U$.
If two vertices $u, v\in V(G)$ are $(K_r, m, 1)$-reachable, then we say $u$ is \emph{$1$-reachable} to $v$.
If $u, v\in U$ are $(K_r, m, t)$-reachable, and the corresponding $K_r$-connector for $u, v$ is contained in $U$, then we say that $u, v\in U$ are \emph{$(K_r, m, t)$-inner-reachable}.
Similarly, we can define \emph{$(K_r, m, t)$-inner-closed} and \emph{$1$-inner-reachable}.
\end{defn}

The following result from \cite{HMWY2021} builds a sufficient condition to ensure that every subset $S\subseteq V(G)$ with $|S|=r$ has linearly many vertex-disjoint absorbers.

\begin{lemma}[\cite{HMWY2021}, Lemma 3.9]\label{lem5.2}
Given $\beta>0$ and $t, r\in \mathbb{N}$ with $r\geq 3$, the following holds for sufficiently large $n\in \mathbb{N}$. Let $G$ be an $n$-vertex graph such that $V(G)$ is $(K_r, \beta n, t)$-closed. Then every $S\in \tbinom{V(G)}{r}$ has a family of at least $\frac{\beta}{r^{2}t}n$ vertex-disjoint $(K_r, t)$-absorbers.
\end{lemma}

Based on this lemma, it suffices to show that $V(G)$ is closed. However, we are only able to prove a slightly weaker result which states that the graph $G$ admits a vertex partition $V(G)=B\cup U$ where $B$ is a small vertex set and $U$ is inner-closed.

\begin{lemma}[]\label{lem7.1}
Given $r\in \mathbb{N}$ with $r\ge 4$ and constants $\tau, \mu$ with $0<\tau<\mu$, there exist positive constants $\alpha, \beta$ and $t\in \mathbb{N}$ such that the following holds for sufficiently large $n$.
Let $G$ be an $n$-vertex graph with $\delta(G)\geq \left(\frac{1}{2}+\mu\right)n$ and $\alpha_{r-2}(G)\leq \alpha n$.
Then $G$ admits a partition $V(G)=B\cup U$ with $|B|\leq \tau n$ and $U$ is $(K_{r}, \beta n, t)$-inner-closed.
\end{lemma}

Clearly, we shall focus on the subgraph $G[U]$ and obtain an absorbing set by applying Lemma \ref{lem5.2} and Lemma \ref{lem5.1} on $G[U]$.

The next step is to deal with the vertex set $B$.
We shall pick mutually vertex-disjoint copies of $K_{r}$ each covering a vertex in $B$.
To achieve this, we use the following result.

%which enables us to find linearly many copies of $K_{r}$ covering any fixed vertex.

\begin{lemma}[]\label{lem5.7}
Given  $r\in \mathbb{N}$ with $r\ge 4$ and a constant $\mu>0$, there exists $\alpha>0$ such that the following holds for sufficiently large $n$.
Let $G$ be an $n$-vertex graph with $\delta(G)\geq \left(\frac{1}{2}+\mu\right)n$ and $\alpha_{r-2}(G)\leq \alpha n$.
If $W$ is a subset of $V(G)$ with $|W|\leq \frac{\mu}{2}n$, then there exists at least one copy of $K_{r}$ in $G-W$ covering $v$ for each $v\in V(G)\backslash W$.
\end{lemma}

\begin{proof} [Proof of Lemma \ref{lem5.7}]
Given $r\in \mathbb{N}$ with $r\geq 4$ and $\mu>0$, we choose $\frac{1}{n}\ll\alpha \ll \mu$.
Let $v\in V(G)$ and $W\subset V(G)$ with $|W|\leq \frac{\mu}{2}n$.
%Then $\delta(G')\geq \left(\frac{1}{2}+\frac{\mu}{2}\right)n$ and $\alpha_{\ell}(G')\leq \alpha n$.
In $N(v)\backslash W$, we arbitrarily choose a vertex, say $u$.
Then $|(N(v)\cap N(u))\backslash W|\geq (|N(v)\backslash W|+|N(u)\backslash W|)-|V(G)\backslash W|\geq 2\left(\frac{1}{2}+\mu\right)n-n-|W|\geq \frac{3}{2}\mu n\geq \alpha n\geq \alpha_{r-2}(G)$, as $\alpha\ll \mu$.
Hence, there exists a copy of $K_{r-2}$ in $G[(N(v)\cap N(u))\backslash W]$, which together with $\{v, u\}$ forms a copy of $K_{r}$ as desired.
\end{proof}

Now, we give the proof of Lemma \ref{lem3.3} by using Lemma \ref{lem5.1}, Lemma \ref{lem5.2}, Lemma \ref{lem7.1} and Lemma \ref{lem5.7}.
\begin{proof} [Proof of Lemma \ref{lem3.3}]
Let $r\in \mathbb{N}$ with $r\ge 4$ and constants $\gamma, \mu$ with $0<\gamma\leq \frac{\mu}{2}$.
Then we take $\tau=\frac{\gamma}{2r}$ and
%and choose
\[
\frac{1}{n}\ll \alpha\ll \xi\ll \frac{1}{t}, \beta\ll \gamma, \mu.
\]
Let $G$ be an $n$-vertex graph with $\delta(G)\geq \left(\frac{1}{2}+\mu\right)n$ and $\alpha_{r-2}(G)\leq \alpha n$.
Applying Lemma \ref{lem7.1} on $G$ with $\tau\ll\mu$, we obtain that $G$ admits a vertex partition $V(G)=B\cup U$ such that $|B|\leq \tau n$ and $U$ is $(K_{r}, \beta n, t)$-inner-closed.
Applying Lemma \ref{lem5.2} on $G[U]$, it holds that every $S\in \tbinom{U}{r}$ has a family of at least $\frac{\beta}{r^{2}t}|U|\geq \frac{\beta}{2r^{2}t}n$ vertex-disjoint $(K_{r}, t)$-absorbers.
Applying Lemma \ref{lem5.1} on $G[U]$ where $\frac{\gamma}{2}$ plays the role of $\gamma$, we obtain a $\xi$-absorbing set $A_{1}$ in $G[U]$ of size at most $\frac{\gamma}{2}n$.

Now, we shall iteratively pick vertex-disjoint copies of $K_{r}$ each covering one vertex in $B$ while avoiding using any vertex in $A_{1}$, and we claim that every vertex in $B$ can be covered in this way.
Let $G_{1}:= G-A_{1}$.
For each $v\in B$, we apply Lemma \ref{lem5.7} iteratively to find a copy of $K_{r}$ covering $v$ in $G_{1}$, while avoiding $A_{1}$ and all copies of $K_{r}$ found so far. This is possible as during the process the number of vertices that we need to avoid is at most

\begin{center}
$r|B|+|A_{1}|\leq r\tau n+ \frac{\gamma}{2}n=\gamma n\leq \frac{\mu}{2}n$.
\end{center}
%where we use the fact that $\tau=\frac{\gamma}{2r}$ and $0<\gamma\leq \frac{\mu}{4}$.

Let $W$ be the union of the vertex sets over all the $|B|$ vertex-disjoint copies of $K_{r}$ as above and $A:= A_{1}\cup W$. Recall that $A_{1}$ is a $\xi$-absorbing set for $G[U]$, and $G[W]$ has a $K_{r}$-factor. Thus $A$ is a $\xi$-absorbing set for $G$ with $|A|\leq \gamma n$.
\end{proof}

Now it remains to prove Lemma \ref{lem7.1} whose proof will be given in next subsection.

\subsection{Proof of Lemma \ref{lem7.1}}
To prove Lemma \ref{lem7.1}, we divide the proof into two steps (i): $G$ admits a partition $V(G)=B\cup U$ where $B$ is a small-sized set and every vertex in $U$ is $1$-inner reachable to linearly many vertices; (ii): we show that $U$ is inner-closed.
The following result makes the first step.

\begin{lemma}[]\label{lem5.5}
Given $r\in \mathbb{N}$ with $r\ge 4$ and constants $\tau, \mu$ with $0<\tau<\mu$, there exist positive constants $\alpha, \beta_{1}, \gamma_{1}$ such that the following holds for sufficiently large $n$. Let $G$ be an $n$-vertex graph with $\delta(G)\geq \left(\frac{1}{2}+\mu\right)n$ and $\alpha_{r-2}(G)\leq \alpha n$. Then $G$ admits a vertex partition $V(G)=B\cup U$ such that $|B|\leq \tau n$ and every vertex in $U$ is $(K_{r}, \beta_{1}n, 1)$-inner-reachable to at least $\gamma_{1}n$ other vertices in $G[U]$.
\end{lemma}

In the second step, we only need to apply the following result on $G[U]$.

\begin{lemma}[]\label{lem5.4}
Given $r\in \mathbb{N}$ with $r\ge 4$ and constants $\mu, \beta_{1}, \gamma_{1}$ with $0<\mu, \beta_{1}, \gamma_{1}<1$, there exist positive constants $\alpha, \beta$ and $t\in \mathbb{N}$ such that the following holds for sufficiently large $n$.
Let $G$ be an $n$-vertex graph with $\delta(G)\geq \left(\frac{1}{2}+\mu\right)n$ and $\alpha_{r-2}(G)\leq \alpha n$ such that every vertex in $V(G)$ is $(K_{r}, \beta_{1}n, 1)$-reachable to at least $\gamma_{1}n$ other vertices. Then $V(G)$ is $(K_{r}, \beta n, t)$-closed.
\end{lemma}

%\begin{proof} [Proof of Lemma \ref{lem7.1}]
Obviously, Lemma \ref{lem7.1} is an immediate corollary of the above-mentioned two lemmas. In the following, we will give the proofs of Lemma \ref{lem5.5} and Lemma \ref{lem5.4} respectively.

\subsubsection{Proof of Lemma \ref{lem5.5} : 1-reachability}
The proof of Lemma \ref{lem5.5} goes roughly as follows.
We first apply the regularity lemma on $G$ to obtain a partition and a reduced graph $R$.
The minimum degree of $R$ is large enough to guarantee that every cluster $V_i$ is covered by a copy of $K_3$.
%A result of Knierm and Su \cite{MR4193066} guarantees that every cluster $V_{i}$ is covered by a $K_{5}$-embeddable structure, say $\mathcal{K}_{i}$ in the reduced graph $R$ (See Lemma \ref{lem5.11}).
In this case, for every $V_{i}$ by using Lemma \ref{lem2.17} and Lemma \ref{lem2.171}, we are able to show that almost all vertices in $V_{i}$ are $1$-reachable to linearly many vertices from $V_{i}$, where only a small proportion of left-over vertex in $V_i$ will be included in $B$.

%the bad vertices would be given iteratively at each stage of the process.

As sketched above, the following fact presents a minimum degree of the reduced graph provided the minimum degree of $G$.
\begin{fac}[]\label{fact2.51}
Let $n\in \mathbb{N}$, $0<\varepsilon, \beta\leq \frac{\mu}{10}$ with $\beta\in [0, 1]$, and $G$ be an $n$-vertex graph with $\delta(G)\geq(\frac{1}{2}+\mu)n$. Let $V(G)=V_{0}\cup \cdots \cup V_{k}$ be a vertex partition of $V(G)$ satisfying Lemma \ref{lem2.3} (1)-(5). We denote the reduced graph as $R_{\beta, \varepsilon}$. Then for every $V_{i}\in V(R_{\beta, \varepsilon})$ we have
\begin{center}
$d_{R_{\beta, \varepsilon}}(V_{i})\geq\left(\frac{1}{2}+\frac{\mu}{2}\right)k$.
\end{center}
\end{fac}

\begin{proof}[Proof] Note that $|V_{0}|\leq \varepsilon n$ and $|V_{i}|=m$ for each $i\in [k]$.
Every edge in $R_{\beta, \varepsilon}$ represents less than $m^{2}$ edges in $G'-V_{0}$.
Thus we have
\begin{align*}
    d_{R_{\beta, \varepsilon}}(V_{i}) & \geq \frac{|V_{i}|(\delta(G)-(\beta+\varepsilon)n-\varepsilon n)}{m^{2}} \\
    & \geq \frac{\left(\frac{1}{2}+\mu-2\varepsilon-\beta\right)mn}{m^{2}} \\
    & \geq \left(\frac{1}{2}+\frac{\mu}{2}\right)k,
\end{align*}
since $0<\varepsilon, \beta\leq \frac{\mu}{10}$.
\end{proof}

Next, we prove Lemma \ref{lem5.5}.

\begin{proof} [Proof of Lemma \ref{lem5.5}]Given $r\in \mathbb{N}$ with $r\ge 4$ and constants $\tau, \mu$ with $0<\tau< \mu$, we shall choose

\begin{center}
$\frac{1}{n}\ll\alpha\ll \beta_{1}, \gamma_{1}\ll\frac{1}{k}\ll \varepsilon \ll \tau, \mu$.
\end{center}
Let $\beta=\frac{\mu}{10}$, $G$ be an $n$-vertex graph with $\delta(G)\geq \left(\frac{1}{2}+\mu\right)n$ and $\alpha_{r-2}(G)\leq \alpha n$.
Applying Lemma \ref{lem2.3} on $G$ with $\varepsilon, \beta>0$, we obtain an $\varepsilon$-regular partition $\mathcal{P}=\{V_{0}, V_{1}, \dots, V_{k}\}$ of $G$.
Let $m:=|V_{i}|$ for each $i\in [k]$ and $R:=R_{\beta, \varepsilon}$ be a reduced graph
%of the $\varepsilon$-regular partition
for $\mathcal{P}$.
By Fact \ref{fact2.51}, we have $\delta(R)\geq \left(\frac{1}{2}+\frac{\mu}{2}\right)k$.
%Let $V(R)=\{V_{1}, \dots, V_{k}\}$, and for distinct $i, j\in[k]$ we draw a edge between $V_{i}$ and $V_j$ if $d_{G'}(V_i, V_j)\geq \beta$ where $G'$ comes from Lemma \ref{lem2.3}.
%By the similar proof in Fact \ref{fact2.5}, it holds that $\delta(R)\geq (\frac{1}{2}+\frac{\mu}{2})k$.
Clearly, every cluster is in a copy of $K_3$ in $R$.
Hence, there exists a minimal family of copies of $K_{3}$, say $\{\mathcal{K}_{1},\dots, \mathcal{K}_{\ell}\}$, such that $V(R)=\bigcup_{i=1}^{\ell}$ $V(\mathcal{K}_{i})$.
Obviously, $\ell\leq k$.
Now, we first define the vertex set $B$ as follows:

\begin{itemize}
  \item [(1)]
  let $S_{i}=V(\mathcal{K}_{i})\backslash \bigcup_{p=1}^{i-1} V(\mathcal{K}_{p})$ for $i\in [\ell]$, where $S_{1}=V(\mathcal{K}_{1})$. Then it follows from the minimality that $S_{i}\neq \emptyset$ and we write $S_{i}=\{V_{i_{1}},\dots, V_{i_{s_{i}}}\}$ for some integer $s_{i}:=|S_{i}|$;
  \item [(2)]
  for $i\in [\ell]$ and $j\in [s_{i}]$, $B_{i_{j}}:=\Big\{v\in V_{i_{j}}\Big| |N(v)\cap V_{s}|\leq (d(V_{i_{j}}, V_{s})-\varepsilon)|V_{s}|$ for some $V_{s}\in V(\mathcal{K}_{i})\, \text{with}\, s\neq i_j\Big\}$, and $B_{i}:=\bigcup_{j=1}^{s_{i}} B_{i_{j}}$;
  \item [(3)]
  $B=\bigcup_{i=1}^{\ell} B_{i}$.
\end{itemize}

Observe that $\{S_{1}, S_{2},\dots, S_{\ell}\}$ is a partition of $V(R)$ and $|S_{i}|\le 3$, $i\in [\ell]$. Moreover, for every $i\in [\ell]$ and $j\in [s_{i}]$, we have $|B_{i_{j}}|\le \varepsilon m|\mathcal{K}_{i}|=3\varepsilon m$.
Thus $|B|\leq3\varepsilon mk\leq \tau n$ since $\varepsilon \ll \tau$.

Let $U:=V(G)\backslash B=\bigcup_{i=1}^{k} U_{i}$ where $U_{i}:= V_{i}\backslash B$.
Then $|U_{i}|\geq m-3\varepsilon m$.
By Lemma \ref{lem2.2}, $d(U_{i}, U_{j})\geq d(V_{i}, V_{j})-\varepsilon$ for distinct $i, j\in [k]$.
%, since $\varepsilon\ll \mu$ and $\beta=\frac{\mu}{10}$.
Next, we shall prove that every $v\in U$ is 1-inner-reachable to linearly many vertices.

For each $i\in [k]$ and any vertex $v\in U_{i}$, we choose the minimum $j\in [\ell]$ such that $V_{i}\in V(\mathcal{K}_{j})$.
%Let $\phi_{j}$ be the corresponding $K_{5}$-multi-embedding.
%Since $\mathcal{K}_{q}$ is a $K_{5}$-embeddable structure, there exists a subgraph of $\mathcal{K}_{q}$, say $\mathcal{K}'_{q}$, which is a $K_{4}$-embeddable structure such that $i_{\mathcal{K}'_{q}}(V_{p})=1$.
Without loss of generality, we may assume $i=1$ and write $V(\mathcal{K}_{j})=\{V_{1}, V_{2}, V_{3}\}$.
Thus for distinct $p, q\in [3]$, it follows from (2) and the fact $\varepsilon\ll \beta, \frac{1}{r}$ that every vertex $u\in U_{p}$ has at least $|N(u)\cap V_{q}|-|B\cap V_{q}|\geq d(V_{p}, V_{q})m-\varepsilon m-3\varepsilon m\geq \frac{\beta}{2}m$ neighbors in $U_{q}$.

Recall that $v\in U_{1}$.
We denote by $U_{1}^{\ast}$ be the set of vertices $u\in U_{1}$ such that for every $p\in \{2, 3\}$, $|N(u)\cap N(v)\cap U_{p}|\geq \left(\frac{\beta}{2}\right)^{2}m$.
The following claim would complete our proof because $|U_{1}^{\ast}|\geq |U_1|-2\varepsilon m\geq (1-5\varepsilon)m\geq \gamma_{1}n$, where $\gamma_{1}\ll \frac{1}{k}\ll \varepsilon$.

\begin{claim}
The vertex $v$ is $(K_{r}, \beta_{1}n, 1)$-reachable to every $u\in U^{\ast}_{1}$.
\end{claim}

To prove this, we arbitrary choose a vertex $u\in U^{\ast}_{1}$. Let $N_{p}:=N(u)\cap N(v)\cap U_{p}$ for each $p\in \{2, 3\}$ and $N_{1}=U_{1}\backslash \{u, v\}$.
Then $|N_{p}|\geq \left(\frac{\beta}{2}\right)^{2}m$ for each $p\in [3]$.
For each $p\in [3]$, $N'_{p}$ comes from $N_{p}$ by deleting any $\beta_{1}n$ vertices.
Hence, $|N_{p}'|\geq |N_{p}|-\beta_{1}n\geq\left(\frac{\beta}{2}\right)^{2}m-\beta_{1}n\geq\frac{\beta^{2}}{8}m$ since $\beta_{1}\ll \beta$.
By Lemma \ref{lem2.2}, $(N'_{p}, N'_{q})$ is $\varepsilon'$-regular with $\varepsilon':=\max\{2\varepsilon, \frac{|V_{q}|\varepsilon}{|N'_{q}|}\}\leq\frac{8\varepsilon}{\beta^{2}}$ and $d(N'_{p}, N'_{q})\geq d(V_{p}, V_{q})-\varepsilon$ for distinct $p, q\in [3]$.
Hence, there exists a vertex, say $w\in N'_2$, with $|N(w)\cap N'_3|\geq (\beta-\varepsilon')|N'_3|\geq \frac{\beta}{2}|N'_3|\geq \frac{\beta^3}{16}m\geq \alpha n$, as $\frac{1}{n}\ll\alpha\ll\varepsilon\ll\mu$.
As $|N(w)\cap N'_3|\geq \alpha n$, $G[N(w)\cap N'_3]$ contains a copy of $K_{r-2}$, say $H$.
Note that $H$ together with $w$ forms a copy of $K_{r-1}$ which lies in $N'_{2}\cup N'_{3}\subseteq N(v)\cap N(u)$.
Hence, $v$ is $(K_{r}, \beta_{1}n, 1)$-reachable to $u$.
%\red{Applying Lemma \ref{lem2.171} on $G[N'_{1}, N'_{2}, N'_{3}]$, we obtain a copy of $Q_{2}$.
%In $Q_{2}$, there exists a copy of $K_{r}$ with $r-1=\ell+1$ vertices in $N'_{2}\cup N'_{3}$.
%Since $N'_{2}, N'_{3}\subseteq N(v)\cap N(u)$, $v$ is $(K_{r}, \beta_{1}n, 1)$-reachable to $u$.}
\end{proof}

\subsubsection{Proof of Lemma \ref{lem5.4}}

In the following, we shall use the latticed-based absorbing method developed by Han \cite{MR3632565} and begin with the following notion introduced by Keevash and Mycroft \cite{MR3290271}. Let $G$ be an $n$-vertex graph.
We will often work with a vertex partition $\mathcal{P}=\{V_{1}, \dots, V_{k}\}$ of $V(G)$ for some integer $k\geq 1$.
For any subset $S\subseteq V(G)$, the \emph{index vector} of $S$ with respect to $\mathcal{P}$, denoted by $i_{\mathcal{P}}(S)$, is the vector in $\mathbb{Z}^{k}$ whose $i$th coordinate is the size of the intersections of $S$ with $V_{i}$ for each $i\in [k]$. For each $j\in [k]$, let $\textbf{u}_{j}\in \mathbb{Z}^{k}$ be the $j$th unit vector, i.e. $\textbf{u}_{j}$ has 1 on the $j$th coordinate and 0 on the other coordinates.
A \emph{transferral} is a vector of the form $\textbf{u}_{i}-\textbf{u}_{j}$ for some distinct $i, j\in [k]$.
A vector $\textbf{v}\in \mathbb{Z}^{k}$ is an \emph{$s$-vector} if all its coordinates are non-negative and their sum is $s$.
Given $\mu>0$ and $r\in \mathbb{N}$, a $r$-vector \textbf{v} is called \emph{$(K_r, \mu)$-robust} if for any set $W$ of at most $\mu n$ vertices, there is a copy of $K_r$ in $G-W$ whose vertex set has an index vector $\textbf{v}$.
Let $I^{\mu}(\mathcal{P})$ be the set of all $(K_r, \mu)$-robust $r$-vectors and $L^{\mu}(\mathcal{P})$ be the lattice (i.e. the additive subgroup) generated by $I^{\mu}(\mathcal{P})$.

Here is a brief proof outline for Lemma \ref{lem5.4}.
In order to prove that $V(G)$ is closed, we adopt a less direct approach and build on the merging techniques developed in \cite{HMWY2021}.
We first partition $V(G)$ into a constant number of parts each of which is closed (see Lemma \ref{lem5.8}). Then we try to merge some of them into a larger (still closed) part by analyzing the graph structures. Lemma \ref{lem5.9} allows us to iteratively merge two distinct parts into a closed one, given the existence of a transferral. Therefore, the key step is to find a transferral (see Lemma \ref{lem5.10}), where we shall use the regularity method, Lemma \ref{lem2.17} and Lemma \ref{lem2.171}.

The following lemma can be used to construct a partition such that each part is closed.

\begin{lemma}[\cite{HMWY2021}, Lemma 3.10]\label{lem5.8}
For any positive constants $\gamma_{1}, \beta_{1}$ and $r\in \mathbb{N}$ with $r\geq 3$, there exist $\beta_{2}=\beta_{2}(\gamma_{1}, \beta_{1}, r)>0$ and $t_{2}\in \mathbb{N}$ such that the following holds for sufficiently large n.
Let $G$ be an $n$-vertex graph such that every vertex in $V(G)$ is $(K_{r}, \beta_{1} n, 1)$-reachable to at least $\gamma_{1} n$ other vertices.
Then there is a partition $\mathcal{P}=\{V_{1}, \dots, V_{p}\}$ of $V(G)$ with $p\leq \lceil\frac{1}{\gamma_{1}}\rceil$ such that for each $i\in [p]$, $V_{i}$ is $(K_{r}, \beta_{2} n, t_{2})$-closed and $|V_{i}|\geq \frac{\gamma_{1}}{2}n$.
\end{lemma}

%The following result gives a sufficient condition that allows us to iteratively merge two distinct parts into a closed one, given the existence of a transferral.

\begin{lemma}[\cite{HMWY2021}, Lemma 4.4]\label{lem5.9}
Given $t, r\in \mathbb{N}$ and constant $\beta>0$, the following holds for sufficiently large $n$.
Let $G$ be an $n$-vertex graph with a partition $\mathcal{P}=\{V_{1}, \dots, V_{p}\}$ of $V(G)$ such that each $V_{i}$ is $(K_{r}, \beta n, t)$-closed.
For distinct $i, j\in [p]$, if there exist two $r$-vectors $\textbf{s}, \textbf{t}\in I^{\beta}(\mathcal{P})$ such that $\textbf{s} - \textbf{t}=\textbf{u}_{i} - \textbf{u}_{j}$, then $V_{i}\cup V_{j}$ is $(K_{r}, \frac{\beta n}{2}, 2rt)$-closed.
\end{lemma}

Note that to invoke Lemma \ref{lem5.9}, we need the following lemma which provides a sufficient condition for the existence of a transferral.

\begin{lemma}[]\label{lem5.10}
Given $p, r\in \mathbb{N}$ and constants $\mu, \delta_{1}>0$, there exist $\alpha, \beta'>0$ such that the following holds for sufficiently large $n$.
Let $G$ be an $n$-vertex graph with $\delta(G)\geq (\frac{1}{2}+\mu)n$, $\alpha_{r-2}(G)\leq \alpha n$, $\mathcal{P}=\{V_{1}, \dots, V_{p}\}$ be a partition of $V(G)$ with $|V_{i}|\geq \delta_{1}n$ for each $i\in [p]$.
If $p\geq 2$, then there exist two $r$-vectors $\textbf{s}, \textbf{t}\in I^{\beta'}(\mathcal{P})$ such that $\textbf{s} - \textbf{t}=\textbf{u}_{i} - \textbf{u}_{j}$ for some distinct $i, j\in [p]$.
\end{lemma}

Now, we have collected all the tools for the proof of Lemma \ref{lem5.4}.

\begin{proof} [Proof of Lemma \ref{lem5.4}] Given $r\in \mathbb{N}$ and constants $\beta_{1}, \gamma_{1}, \mu>0$, we choose

\begin{center}
$\frac{1}{n}\ll\alpha\ll\beta, \frac{1}{t}\ll \beta_{2}, \frac{1}{t_{2}}\ll \beta_{1}, \gamma_{1}, \mu$.
\end{center}
Let $G$ be an $n$-vertex graph with $\delta(G)\geq (\frac{1}{2}+\mu)n$, $\alpha_{r-2}(G)\leq \alpha n$ and every vertex in $V(G)$ is $(K_{r}, \beta_{1}n, 1)$-reachable to at least $\gamma_{1}n$ other vertices.
Applying Lemma \ref{lem5.8} on $G$, we obtain a partition $\mathcal{P}_{0}=\{V_{1}, \dots, V_{p}\}$ for some $p\leq \lceil\frac{1}{\gamma_{1}}\rceil$, where each $V_{i}$ is $(K_{r}, \beta_{2}n, t_{2})$-closed and $|V_{i}|\geq \frac{\gamma_{1}n}{2}$.

Let $\mathcal{P}'=\{U_{1}, \dots, U_{p'}\}$ be a vertex partition of $G$ with minimum $|\mathcal{P}'|$ such that each $|U_{i}|\geq \frac{\gamma_{1}n}{2}$ and $U_{i}$ is $(K_{r}, \beta n, t)$-closed.
We claim that $p'=1$.
If $p'\geq 2$, then by Lemma \ref{lem5.9} and Lemma \ref{lem5.10}, there exist two distinct vertex parts $U_{i}$ and $U_{j}$ for distinct $i, j\in [p']$ such that $U_{i}\cup U_{j}$ is $(K_{r}, \beta'n, t')$-closed for some $\beta'$ and $t'$.
By taking $U_{i}\cup U_{j}$ as a new part in partition and renaming all the parts if necessary, we get a partition $\mathcal{P}''$ with $|\mathcal{P}''|<|\mathcal{P}'|$, which contradicts the minimality of $|\mathcal{P}'|$.
Hence, $V(G)$ is $(K_{r}, \beta n, t)$-closed.
\end{proof}

Next, we will give a proof of Lemma \ref{lem5.10}.
In order to prove Lemma \ref{lem5.10}, we use the regularity lemma (Lemma \ref{lem2.3}) and some embedding results.
% (Lemma \ref{lem2.17} and Lemma \ref{lem2.171}).
In particular, these embedding results allow us to construct vertex-disjoint copies of $K_{r}$ with different vertex distribution, which can be used to show the existence of a transferral.
%This roughly reduces the problem to finding in the reduced graph a $K_{5}$-embeddable structure.

\begin{proof}[Proof of Lemma \ref{lem5.10}] Given $p, r, t\in \mathbb{N}$ and positive constants $\mu, \delta_{1}$, we choose
\begin{center}
$\frac{1}{n}\ll \alpha\ll \frac{1}{N}\ll\beta'\ll \frac{1}{k}\ll \varepsilon\ll \mu, \delta_{1}$.
\end{center}
Let $\beta=\frac{\mu}{10}$, $G$ be an $n$-vertex graph with $\delta(G)\geq (\frac{1}{2}+\mu)n$, and $\mathcal{P}=\{V_{1}, \dots, V_{p}\}$ be a vertex partition of $V(G)$ with $|V_{i}|\geq \delta_{1}n$ for each $i\in [p]$.
Anchoring at the current vertex partition of $V(G)$, we apply Lemma \ref{lem2.3} with $\varepsilon, \beta>0$ and refine the current partition.
After refinement, we denote the $\varepsilon$-regular partition by $\mathcal{P}'=\{V_{0}, V_{1, 1}, \dots, V_{1, s_{1}}, \dots, V_{p, 1}, \dots, V_{p, s_{p}}\}$ where $V_{i, j}\subseteq V_{i}$ and $s_{i}\in \mathbb{N}$ for each $i\in [p]$, $j\in [s_{i}]$.
Let $R:=R_{\beta, \varepsilon}$ be the reduced multigraph with multiplicity 2.
Let $|V(R)|=k$ and $\mathcal{V}_{i}:=\{V_{i, 1}, \dots, V_{i, s_{i}}\}$ be a vertex subset of $V(R)$ for each $i\in[p]$.
We may assume $|\mathcal{V}_{i}|\leq |\mathcal{V}_{i+1}|$ for each $i\in [p-1]$.
By Fact \ref{fact2.5}, it holds that $\delta(R)\geq (1+\mu)k$.
Clearly, every cluster is in a double-edge.

We call a subgraph $\mathcal{K}\subseteq R$ \emph{crossing} with respect to the partition $\mathcal{P}$ if $V(\mathcal{K})\cap \mathcal{V}_{i}\neq\emptyset$ and $V(\mathcal{K})\cap \mathcal{V}_{j}\neq\emptyset$ for distinct $i, j\in [p]$.
For every cluster $V_{i, j}$, $D_{i, j}$ denotes the double-edge neighborhood of $V_{i, j}$.
%and $S'_{i, j}$ denotes the single-edge neighborhood of $V_{i, j}$.

We first claim that there is no crossing double-edge.
Assume for contradiction, $V_{i, 1}V_{j, 1}$ is a crossing double-edge in $R$ for some distinct $i, j\in [p]$.
As $\delta(R)\geq (1+\mu)k$, it holds that $N(V_{i, 1})\cap N(V_{j, 1})\neq\emptyset$.
Assume $V_{w,q}\in N(V_{i, 1})\cap N(V_{j, 1})$ and $V_{w,q}\in \mathcal{V}_{w}$ for some $w\in [p]$ and $q\in [s_{w}]$.
Here, we further assume $w\neq i$.
The proof of the case $w=i$ is similar, and we omit it.
As $V_{i, 1}V_{j, 1}$ is a double-edge in $R$, it holds that $(V_{i, 1}, V_{j, 1})$ is $\varepsilon$-regular with $d(V_{i, 1}, V_{j, 1})\geq \frac{1}{2}+\beta$.
Let $V'_{i, 1}$ (and $V'_{j, 1}$) comes from $V_{i, 1}$(and $V_{j, 1}$) by deleting any $\beta'n$ vertices, thus $|V'_{i, 1}|~(\text{and}~|V'_{j, 1}|)\geq m-\beta'n\geq\frac{\beta}{2}m$, since $\beta'\ll \mu$ and $\beta=\frac{\mu}{10}$.
Then by Lemma \ref{lem2.2} $(V'_{i, 1}, V'_{j, 1})$ is $\varepsilon_{1}$-regular with $\varepsilon_{1}:=\max\left\{2\varepsilon, \frac{\varepsilon|V_{j, 1}|}{|V'_{j, 1}|}\right\}\leq \frac{2}{\beta}\varepsilon$, and $d(V'_{i, 1}, V'_{j, 1})\geq d(V_{i, 1}, V_{j, 1})-\varepsilon\geq \frac{1}{2}+\frac{\beta}{2}$ since $\varepsilon\ll \mu$ and $\beta=\frac{\mu}{10}$.
Every vertex in $V'_{i, 1}$, except at most $\varepsilon_1|V'_{i, 1}|$ vertices, has at least $(\frac{1}{2}+\beta-\varepsilon_1)|V'_{j, 1}|\geq (\frac{1}{2}+\frac{\beta}{2})|V'_{j, 1}|$ neighbors in $V'_{j, 1}$, as $\varepsilon\ll\mu$.
Since $|V'_{i, 1}|\geq \alpha_{r-2}(G)+\varepsilon_1|V'_{i, 1}|$, there exists a copy of $K_{r-2}$, say $H$, in $V'_{i, 1}$ so that every vertex in it has at least $(\frac{1}{2}+\frac{\beta}{2})|V'_{j, 1}|$ neighbors in $V'_{j, 1}$.
In $H$, we choose an edge, say $uv$.
It holds that $|N(u)\cap N(v)\cap V'_{j, 1}|\geq 2(\frac{1}{2}+\frac{\beta}{2})|V'_{j, 1}|-|V'_{j, 1}|\geq \beta|V'_{j, 1}|\geq \frac{\beta^2}{2}m\geq \alpha n$, as $\frac{1}{n}\ll\alpha\ll\beta$.
As $\alpha_{r-2}(G)\leq \alpha n$, $G[N(u)\cap N(v)\cap V'_{j, 1}]$ contains a copy of $K_{r-2}$ which together with $\{u, v\}$ forms a copy of $K_r$, say $H_1$, with $|V(H_1)\cap V'_{i, 1}|=2$ and $|V(H_1)\cap V'_{j, 1}|=r-2$.
This gives us a $r$-vector $\textbf{s}\in I^{\beta'}(\mathcal{P})$.

Since $V_{w, q}\in N(V_{i, 1})\cap N(V_{j, 1})$, it holds that $(V_{i, 1}, V_{j, 1})$ is $(\varepsilon, \beta)$-regular, $(V_{i, 1}, V_{w, q})$ is $(\varepsilon, \beta)$-regular and $(V_{j, 1}, V_{w, q})$ is $(\varepsilon, \beta)$-regular.
%By the similar argument as above, we can obtain a copy of $Q_{2}$ in $G[V'_{i, 1}, V'_{j, 1}, V'_{w,q}]$ where
Let $V'_{i, 1}, V'_{j, 1}$ and $V'_{w, q}$ come from $V_{i+1, 1}, V_{i, 1}$ and $V_{w, q}$ by deleting any $\beta'n$ vertices respectively.
Thus all but at most $2\varepsilon m+\beta'n\leq 3\varepsilon m$ vertices in $V_{i, 1}$ so that every vertex in it has at least $(\beta-\varepsilon)m-\beta'n\geq \frac{\beta}{4}m$ in $V'_{j, 1}$ and $V'_{w, q}$ respectively.
We choose such a vertex, say $v$, with $v\in V'_{i, 1}$, $|N(v)\cap V'_{j, 1}|\geq \frac{\beta}{4}m$ and $|N(v)\cap V'_{w, q}|\geq \frac{\beta}{4}m$.
By the similar arguments as above, we are able to obtain a copy of $K_{r-1}$ with one vertex in $N(v)\cap V'_{w, q}$ and $r-2$ vertices in $N(v)\cap V'_{j, 1}$ which together with $\{v\}$ forms a copy of $K_r$, say $H_2$.
Note that $H_2$ gives us a $r$-vector $\textbf{t}\in I^{\beta'}(\mathcal{P})$.
We obtain two $r$-vectors $\textbf{s}, \textbf{t}\in I^{\beta'}(\mathcal{P})$ such that $\textbf{s} - \textbf{t}=\textbf{u}_{i} - \textbf{u}_{w}$.
In the following, we may further assume that there is no crossing double-edge.

Next, we claim that there is no crossing $K^{=}_{3}$.
Assume $V_{i,1}V_{i,2}V_{j,1}$ is a crossing $K^{=}_{3}$ for some distinct $i, j\in [p]$.
Clearly, $V_{i,1}V_{i,2}$ is the double-edge.
Let $V'_{i,1}, V'_{i,2}$ and $V'_{j,1}$ come from $V_{i,1}, V_{i,2}$ and $V_{j,1}$ by deleting any $\beta'n$ vertices.
As $\frac{1}{n}\ll\alpha\ll\beta'\ll\frac{1}{k}$, it holds that $|V'_{i,1}|\geq |V_{i,1}|-\beta'n\geq \alpha n\geq \alpha_{r-2}(G)$.
Therefore, there exists a copy of $K_{r-2}$, say $H_3$, in $G[V'_{i,1}]$.
In $H_3$, we take an edge, say $uv\in E(H_3)$.
It holds that $|N(u)\cap N(v)\cap V'_{i, 2}|\geq 2(\frac{1}{2}+\beta)|V_{i, 2}|-|V_{i, 2}|-\beta'n\geq \frac{\beta}{2}m\geq \alpha n$, as $\alpha\ll\beta'\ll\beta$.
Hence, $D[N(u)\cap N(v)\cap V'_{i, 2}]$ contains a copy of $K_{r-2}$ which together with $uv$ forms a copy of $K_r$.
%Similarly, by applying Lemma \ref{lem2.17} on $G[V'_{i,1}, V'_{i,2}]$, there exists a copy of $Q_{1}$.
%In $Q_{1}$,
%there exists a copy of $K_{r}$, say $H_{3}$, with $|V(H_{3})\cap V_{i}|=r$.
This gives us a $r$-vector $\textbf{s}\in I^{\beta'}(\mathcal{P})$.

Since $(V_{i, 1}, V_{j, 1})$ and $(V_{i, 2}, V_{j, 1})$ are $(\varepsilon, \beta)$-regular respectively, all but at most $2\varepsilon|V_{j, 1}|+\beta'n\leq 3\varepsilon m$ vertices in $V_{j, 1}$ lie in $V'_{j, 1}$ so that each of them has at least $(\beta-\varepsilon)m-\beta'n\geq \frac{\beta}{2}m-\beta'n\geq \frac{\beta}{4}m$ vertices in $V'_{i, 1}$ and $V'_{i, 2}$, as $\frac{1}{n}\ll\beta'\ll\varepsilon\ll\mu$.
We choose such a vertex in $V'_{j, 1}$, say $w$, with $|N(w)\cap V'_{i, 1}|\geq \frac{\beta}{4}m$ and $|N(w)\cap V'_{i, 2}|\geq \frac{\beta}{4}m$.
By the similar arguments as above, $G[N(w)\cap V'_{i, 1}, N(w)\cap V'_{i, 2}]$ contains a copy of $K_{r-1}$ with one vertex in $N(w)\cap V'_{i, 1}$ and $r-2$ vertices in $N(w)\cap V'_{i, 2}$.
Till now, we obtain a copy of $K_r$ with one vertex in $V'_{j, 1}$ and $r-1$ vertices in $V'_{i, 1}\cup V'_{i, 2}$.
%Applying Lemma \ref{lem2.171} on $G[V'_{i,1}, V'_{i,2}, V'_{j,1}]$, there exists a copy of $Q_{2}$.
%In $Q_{2}$, there exists a copy of $K_{r}$, say $H_{4}$, with $|V(H_{4})\cap V_{i}|=r-1$ and $|V(H_{4})\cap V_{j}|=1$.
This gives us a $r$-vector $\textbf{t}\in I^{\beta'}(\mathcal{P})$.
We obtain two $r$-vector $\textbf{s}, \textbf{t}\in I^{\beta'}(\mathcal{P})$ such that $\textbf{s} - \textbf{t}=\textbf{u}_{i} - \textbf{u}_{j}$.
In the following, we may further assume that there is no crossing $K^{=}_{3}$.

Assume $V_{i,1}V_{j,1}$ is a crossing single-edge for some distinct $i, j\in [p]$.
As there is no crossing $K^{=}_{3}$, it holds that $N(V_{i,1})\cap D_{j,1}=\emptyset$ and $N(V_{j, 1})\cap D_{i, 1}=\emptyset$.
Hence, we have $d(V_{i, 1})\leq k-|D_{i, 1}|-|D_{j, 1}|+2|D_{i, 1}|$ and $d(V_{j, 1})\leq k-|D_{i, 1}|-|D_{j, 1}|+2|D_{j, 1}|$ which yields $(2+2\mu)k\leq2\delta(R)\leq d(V_{i, 1})+d(V_{i+1, 1})\leq2k$, a contradiction.
\end{proof}

\bibliographystyle{abbrv}
\bibliography{ref}

%\section*{Acknowledgement}
%We would like to thank Michael Krivelevich for bringing~\cite{KrLS} to our attention.

\end{document}